\documentclass{amsart}
\usepackage{amssymb,latexsym}
\usepackage{color}
\theoremstyle{plain}
\newtheorem{theorem}{Theorem}
\newtheorem{corollary}{Corollary}

\newtheorem{lemma}{Lemma}
\newtheorem{conjecture}{Conjecture}
\theoremstyle{definition}

\newtheorem{remark}{Remark}

\DeclareMathOperator{\Res}{Res}

\newcommand{\enm}[1]{\ensuremath{#1}} 

\newcommand{\PP}{\enm{\mathbb {P}}}
\newcommand{\Oo}{\enm{\mathcal {O}}}
\newcommand{\Ii}{\enm{\mathcal {I}}}
\newcommand{\KK}{\enm{\mathbb {K}}}

\begin{document}

\title{Maximal rank of space curves in the range A}
\author{Edoardo Ballico, Philippe Ellia, Claudio Fontanari}
\address[Edoardo Ballico, Claudio Fontanari]{Dipartimento di Matematica, Universit\`a di Trento, Via Sommarive 14, 38123 Povo (Trento), Italy.}
\email{edoardo.ballico@unitn.it, claudio.fontanari@unitn.it}
\address[Philippe Ellia]{Dipartimento di Matematica e Informatica, Universit\`a degli Studi di Ferrara, Via Machiavelli 30, 44121 Ferrara, Italy.} 
\email{phe@unife.it}
\thanks{This research was partially supported by PRIN 2012 "Geometria delle variet\`a algebriche",
by FIRB 2012 "Moduli spaces and Applications", and by GNSAGA of INdAM (Italy).}

\begin{abstract}
We prove the following statement, which has been conjectured since 1985: 
\emph{There exists a constant $K$ such that for all natural numbers $d,g$ with $g\le Kd^{3/2}$ there exists an irreducible component of the Hilbert scheme 
of $\PP^3$ whose general element is a smooth, connected curve of degree $d$ and genus $g$ of maximal rank.}
\end{abstract}

\maketitle

\section{Introduction}

The postulation of algebraic space curves has been the object of wide interest in the last thirty years (see for instance \cite{be2}, \cite{be3}, \cite{l}, \cite{jp}, \cite{la}). 
In particular, the following Conjecture was stated in 1985 in \cite{be3}, p. 2 (see also \cite{be}, \S 6, Problem~4): 

\begin{conjecture}\label{main}
There exists a constant $K$ such that for all natural numbers $d,g$ with $g\le Kd^{3/2}$ there exists an irreducible component of the Hilbert scheme of $\PP^3$ whose general element is a smooth, connected curve of degree $d$ and genus $g$ of maximal rank.
\end{conjecture}

\noindent
We recall that a space curve $C$ is of \emph{maximal rank} if the natural maps $H^0(\mathcal{O}_{\PP^3}(m)) \to H^0(\mathcal{O}_C(m))$ are either injective or surjective for every $m$.

Here we consider smooth and connected curves $X$ with $h^1(\Ii _X(m)) =0$, $h^0(\Ii _X(m-1))=0$, $\deg (X) =d$, 
$g(X)=g$ and $h^1(\Oo _X(m-2)) =0$ (hence of maximal rank by Castelnuovo-Mumford regularity). Since $h^1(\Ii _X(m)) =0$ and $h^1(\Oo _X(m)) =0$, we have

 \begin{equation}\label{eqi1}
 1 +md-g \le \binom{m+3}{3}
 \end{equation}

 Let $d(m,g)_{\mathrm{max}}$ be the maximal integer $d$ such that (\ref{eqi1}) is satisfied, i.e. set $d(m,g)_{\mathrm{max}}$ $:= \lfloor \binom{m+3}{3}+g-1)/m\rfloor$.
 Since $h^0(\Ii _X(m-1)) =0$ and $h^1(\Oo _X(m-1)) =0$, we have

\begin{equation}\label{eqi2}
 1 + (m-1)d-g \ge \binom{m+2}{3}
 \end{equation}

Let $d(m,g)_{\mathrm{min}}$ be the minimal integer $d$ such that (\ref{eqi2}) is satisfied, i.e. set $d(m,g)_{\mathrm{min}}:= \lceil \binom{m+2}{3}+g-1)/(m-1)\rceil$.
 
For every integer $s >0$ define the number $p_a(C_s): = s(s+1)(2s-5)/6 +1$ (which is going to to be the genus of the 
curve $C_s$ to be introduced later in Section~\ref{preliminaries}). For all positive integers $m \ge 3$ set 
\begin{eqnarray*}
\varphi (m) &=& p_a(C_{\lfloor m/\sqrt{20}\rfloor -4}) +p_a(C_{\lfloor m/\sqrt{20}\rfloor -5}) \\
&=&
\frac{(\lfloor m/\sqrt{20}\rfloor -4)(\lfloor m/\sqrt{20}\rfloor -3) (2\lfloor m/\sqrt{20}\rfloor -13)}{6}+1 \\
& & + \frac{(\lfloor m/\sqrt{20}\rfloor -5)(\lfloor m/\sqrt{20}\rfloor -4) (2\lfloor m/\sqrt{20}\rfloor -15)}{6}+1.
\end{eqnarray*}
For any smooth curve $X \subset \PP^3$ let $N_X$ denote the normal bundle of $X$ in $\PP^3$. If $h^1(N_X)=0$, then $X$ is a smooth point of the Hilbert scheme
of $\PP^3$ and this Hilbert scheme has the expected dimension $h^0(N_X)$ at $X$.

Our main result is the following:
\begin{theorem}\label{bingo}
For every integer $m \ge 3$ and every $(d,g)$ with $17052 \le g \le \varphi (m)$ and $d(m,g)_{\mathrm{min}} \le d \le d(m,g)_{\mathrm{max}}$ 
there exists a component of the Hilbert scheme of curves in $\PP^3$ of genus $g$ and degree $d$, whose general element $X$ is smooth and satisfies 
$h^0(\Ii _{X}(m-1)) =0$, $h^1(\Ii _{X}(m)) =0$, $h^1(\Oo _{X}(m-2)) =0$, and $h^1(N_X(-1))=0$. 
\end{theorem}

As an application of Theorem \ref{bingo} we prove Conjecture \ref{main}. Indeed, if $g=0$  we have just to quote \cite{H}. Next, if $0 < g < 17052$ we may choose $K>0$ 
such that $g \ge K (g+3)^{3/2}$. Hence from $K (g+3)^{3/2} \le g \le Kd^{3/2}$ we get $d \ge g+3$ and we are done by \cite{be2}. Finally, if $g \ge 17052$ we have the following:

\begin{corollary}\label{cor}
Let $K = \frac{2}{3} \left(\frac{1}{10} \right)^{3/2}$  and $\varepsilon = \frac{11}{20} +4 \left(\frac{1}{20} \right)^{3/2}$. 
If $17052 \le g \le Kd^{3/2} - 6 \varepsilon d$ then there exists an irreducible component of the Hilbert scheme of $\PP^3$ whose general 
element $X$ is a smooth, connected curve of degree $d$ and genus $g$ of maximal rank and with $h^1(N_X(-1))=0$.
\end{corollary}

The constant $K$ in Corollary \ref{cor} is certainly not optimal, but the exponent $d^{3/2}$ is sharp among 
the curves with $h^1(N_X)=0$ (see \cite{eh}, \cite[Corollaire 5.18]{pe} and \cite[II.3.6]{h2} for the condition $h^1(N_X(-2)) =0$, \cite[II.3.7]{h2} and \cite{w} for the condition $h^1(N_X(-1)) =0$, and \cite[II.3.8]{h2} for the condition $h^1(N_X) =0$).

If $X$ is as in Theorem \ref{bingo}, then by Castelnuovo-Mumford regularity we have $h^1(\Ii _X(t)) =0$ for all $t>m$ and the homogeneous ideal of $X$ is generated by forms of degree $m$ and degree $m+1$. A smooth curve $Y\subset \PP^3$ with $h^0(\Ii _Y(m-1)) =0$, $\frac{m^2+4m+6}{6} \le \deg (Y) <  \frac{m^2+4m+6}{3}$ and maximal genus among the curves with $h^0(\Ii_Y(m-1)) =0$ satisfies $h^1(\Oo_Y(m-1)) =0$ (\cite[proof of Theorem 3.3 at p. 97]{h0}).
In the statement of Theorem \ref{bingo} we claim one shift more, namely, $h^1(\Oo_X(m-2)) =0$, in order to apply Castelnuovo-Mumford regularity to $X$. 

We describe here one of the main differences with respect to \cite{H, be2, be3}. Fix integers $d, g$ as in Theorem \ref{bingo} or Corollary \ref{cor}. 
Suppose that we have constructed two irreducible and generically smooth components $W_1, W_2$ of the Hilbert scheme of smooth space curves of degree $d$ 
and genus $g$.
Suppose also that we have proved the existence of $Y_1\in W_1$ and $Y_2\in W_2$ with $h^0(\Ii _{Y_2}(m-1)) =0$, $h^1(\Ii _{Y_1}(m)) =0$ and 
$h^1(N_{Y_i})=h^1(\Oo _{Y_i}(m-3)) =0$, $i=1, 2$. If $W_1=W_2$, then by the semicontinuity theorem for cohomology and Castelnuovo-Mumford regularity a general 
$X\in W_1$ satisfies $h^0(\Ii _X(m-1))=0$, $h^1(\Ii _X(t)) =0$ for all $t\ge m$ and $h^1(N_X)=0$. In particular a general element of $W_1$ has maximal rank.  
But we need to know that $W_1=W_2$. If $d\ge g+3$ it was not known at that time that the Hilbert scheme of smooth space curves of degree
$d$ and genus $g$ is irreducible (\cite{ein}), but it was obvious since at least Castelnuovo that its part parametrizing the non-special curves is irreducible 
(modulo the irreducibility of the moduli scheme $\mathcal {M}_g$ of genus $g$ smooth curves). When $d<g+3$, the Hilbert scheme of smooth space curves of 
degree $d$ and genus $g$ is often reducible, even in ranges with $d/g$ not small (\cite{d, i, k, kk1, kkl, kk2}). In \cite{be3} when $d\ge (g+2)/2$ we defined 
a certain irreducible component $Z(d,g)$ of the Hilbert scheme of smooth space curves of degree $d$ and genus $g$ and (under far stronger assumptions on $d, g$) 
we were able find $Y_1$ and $Y_2$ with $W_1=W_2=Z(d,g)$. Several pages of Section \ref{Sm} are devoted to solve this problem.

We work over an algebraically closed field $\KK$ of characteristic zero.

We thank the anonymous referee for useful comments.


\section{Preliminaries}\label{preliminaries}

\subsection{The curves $C_{t,k}$}

For each locally Cohen-Macaulay curve $C\subset \mathbb {P}^3$ the index of speciality $e({C})$ of $C$ is the maximal integer $e$ such that
$h^1(\mathcal {O} _C(e)) \ne 0$.

Fix an integer $s>0$. Let $C_s \subset \mathbb {P}^3$ be any curve fitting in an exact sequence
\begin{equation}\label{eqa6}
0 \to \mathcal {O} _{\mathbb {P}^3}(-s-1) \to (s+1)\mathcal {O} _{\mathbb {P}^3}(-s) \to \mathcal {I} _{C_s} \to 0
\end{equation}
Each $C_s$ is arithmetically Cohen-Macaulay and in particular $h^0(\mathcal {O} _{C_s})=1$. By taking the Hilbert function in (\ref{eqa6}) we
get $\deg (C_s) = s(s+1)/2$, $p_a(C_s) = s(s+1)(2s-5)/6 +1$ and $e(C_s) = s-3$. Hence $h^i(\mathcal {I}_{C_s}(s-1)) =0$, $i=0,1,2$. By taking $d:= \deg (C_s)$ we get $p_a(C_s) = 1+d(s-1)-\binom{s+2}{3} = G_A(d,s)$. 
The set of all curves fitting in (\ref{eqa6}) is an irreducible variety and its general member is smooth and connected.
Among them there are the stick-figures called $\mathbf{K}_s$ in \cite{fl1}, \cite{fl2} and \cite{bbem}.  We have $h^1(N_{C_s}(-2)) =0$ for all $C_s$
(\cite[Lemme 1]{Ellia}, see also \cite{e}). Unless otherwise
stated we only use smooth $C_s$. 

For any $t, k$ let $C_{t,k}:= C_t\sqcup C_k$ be the union of a smooth $C_t$ and a smooth $C_k$ with the only restriction that
they are disjoint. By definition each $C_{t,k}$ is smooth. Let $d_{t,k} := \deg (C_{t,k}) =t(t+1)/2 +k(k+1)/2$ and  
$g_{t,k}:=h^1(\Oo_{C_{t,k}}) = 2 +t(t+1)(2t-5)/6+k(k+1)(2k-5)/6$ for $t\ge k >0$.
If $t\ge k>0$ then we have
\begin{equation}\label{eq+bb+1}
(t+k-1)d_{t,k} + 2-g_{t,k} = \binom{t+k+2}{3}
\end{equation}
Since each connected component $A$ of $C_{t,k}$ satisfies $h^i(N_A(-2))=0$, $i=0,1$, we have $h^i(N_{C_{t,k}}(-2)) =0$, $i=0,1$.

\begin{lemma}\label{b1}
We have $h^i(\mathcal {I} _{C_{t,k}}(t+k-1)) =0$, $i=0,1,2$.
\end{lemma}

\begin{proof}
Since $C_t\cap C_k =\emptyset$, we have $Tor^{1}_{{\mathcal {O} _{\mathbb {P}^3}}}(\mathcal {I}_{C_t},\mathcal {I} _{C_k})=0$ and $\mathcal {I} _{C_t}\otimes \mathcal {I} _{C_k} = \mathcal {I}_{C_{t,k}}$.
Therefore tensoring (\ref{eqa6}) with $s:= t$ by $\mathcal {I} _{C_k}(t+k-1)$ we get
\begin{equation}\label{eqa7}
0 \to t\mathcal {I} _{C_k}(k-2) \to (t+1)\mathcal {I} _{C_k}(k-1)\to \mathcal {I} _{C_{t,k}}(t+k-1)\to 0
\end{equation}
We have $h^2(\mathcal {I} _{C_k}(k-2)) = h^1(\mathcal {O} _{C_k}(k-2))$ and the latter integer is zero, because $e(C_k)= k-3<k-2$. We have $h^1(\mathcal {I} _{C_k}(k-1))=0$, because $C_k$
is arithmetically Cohen-Macaulay. We have $h^0(\mathcal {I} _{C_k}(k-1)) =0$, by the case $s=k$ of (\ref{eqa6}). Hence $h^i(\mathcal {I} _{C_{t,k}}(t+k-1)) =0$, $i=0,1,2$.
\end{proof} 

\begin{remark}
In this paper we only need $k\in \{t-1,t\}$.
\end{remark}

\begin{remark}\label{b2}
We have $e(C_{t,k}) =\max \{e(C_t),e(C_k)\} = \max \{t-3,k-3\} \le t+k-4$. Recall that $d_{t,k}= \deg (C_{t,k})$.
If $s:= t+k$, then $d_{s-1,1} = (s^2-s+2)/2 \ge d_{t,k}$. If $s$ is even then $d_{t,k} \ge s(s+2)/4 =d_{\frac{s}{2},\frac{s}{2}}$. If $s$ is odd, then $d_{t,k} \ge (s+1)^2/4 = d_{\frac{s+1}{2},\frac{s-1}{2}}$.
\end{remark}

\begin{remark}\label{ob2.1} Let $X$ be a general smooth curve of genus $g$ and degree $d \ge g+3$ such that
 $h^1(\Oo _X(1)) =0$; if either $g\ge 26$ (\cite[p. 67, inequality $D_P(g) \le g+3$]{pe}) or $g\le 25$ and $d\ge g+14$ (\cite[p. 67]{pe}), then $h^1(N_X(-2)) =0$ 
(\cite{pe} uses the case $g=0$ done in \cite{ev}).
\end{remark}

\subsection{Smoothing}

We are going to apply standard smoothing techniques (see for instance \cite{hh} and \cite{s}).

\begin{lemma}\label{aaa1}
Fix $A\sqcup B $ with $A = C_t$ and $B = C_k$. Let $X$ be a nodal curve with $X = A\cup B\cup Y$, $Y$ a smooth curve of degree $d'\ge 2$ and genus $g'$,
$\sharp (A\cap Y)=1$, $\sharp (B\cap Y)=1$, $h^1(\Oo _Y(1)) =0$ and $h^1(N_Y(-2)) =0$. Then $h^1(N_X(-1))=0$ and $X$ is smoothable.
\end{lemma}

\begin{proof}
Set $C:= A\cup B$. Write $\{p_1\} =A\cap Y$ and $\{p_2\} = B\cap Y$. We have an exact sequence
\begin{equation}\label{eqsss1}
0 \to N_X(-1) \to N_X(-1)|_C \oplus N_X(-1)|_Y \to N_X(-1)|_{\{p_1,p_2\}} \to 0
\end{equation}
Since $N_X(-1)|_C$ is obtained from $N_C(-1)$ by making two positive elementary transformations and $h^1(N_C(-1)) =0$, we have $h^1(N_X(-1)|_C) =0$.
Since $N_X(-2)|_Y$ is obtained from $N_Y(-2)$ by making two positive elementary transformations and $h^1(N_Y(-2)) =0$, we have $h^1(N_X(-2)|_Y) =0$. 
Let $H\subset \PP^3$ be a general plane containing $\{p_1,p_2\}$. Since $Y$ is not a line, $Y\cap H$ is a zero-dimensional scheme.
Since $h^1(N_X(-2)|_Y)=0$, the restriction map 
$$H^0(Y,N_X(-1)|_Y) \to H^0(Y\cap H,N_X(-1)|_{H\cap Y})$$ 
is surjective. Since $\{p_1,p_2\} \subseteq Y\cap H$, the restriction map
$H^0(Y\cap H,N_X(-1)|_{H\cap Y}) \to H^0(\{p_1,p_2\},N_X(-1)|_{\{p_1,p_2\}})$ 
is surjective. Hence the restriction map  
$$H^0(Y,N_X(-1)|_Y) \to H^0(\{p_1,p_2\},N_X(-1)|_{\{p_1,p_2\}})$$ 
is surjective. From (\ref{eqsss1}) we get $h^1(N_X(-1)) =0$.

Since $h^1(N_X(-1)) =0$, $X$ is smoothable (\cite[Corollary 1.2]{fl1}).
\end{proof}

Call $U(t,k,d',g')$ the set of all curves $X = A\cup B\cup Y$ appearing in Lemma \ref{aaa1}. For all integer $y\ge 0$ and $x\ge y+3$ the Hilbert scheme of smooth space curves of degree $x$ and genus $y$ is irreducible (\cite{ein, ein1}). By Lemma \ref{aaa1} there is a unique irreducible component 
$W(t,k,d',g')$ of the Hilbert scheme of $\PP^3$ containing the curve $X$ of Lemma \ref{aaa1}.
A general $C\in W(t,k,d',g')$ is smooth and $h^1(N_C(-1)) =0$. We have $\deg ({C}) = d'+\deg (C_t)+\deg (C_k) = d'+t(t+1)/2 + k(k+1)/2$
and genus $g({C}) = g' +p_a(C_t)+p_a(C_k) = g'-2 +t(t+1)(2t-5)/6 + k(k+1)(2k-5)/6$.


\section{Assertion $M(s,t,k)$, $k\in \{t-1,t\}$}

For any $t\ge 27$, set $c(2t+1,t,t) =t+3$, $d(2t+1,t,t) =0$, $c(2t,t,t-1) =t+2$ and $d(2t,t,t-1) = t-1$. Set $g(t+k+1,t,k):= c(t+k+1,t,k) -3$. 
Note that if $k\in \{t-1,t\}$ we have
\begin{equation}\label{eq+a1}
t(t+1) +k(k+1) +d(t+k+1,t,k) = (t+k)(t+k+4-c(t+k+1,t,k))
\end{equation}

Now fix an integer $s\ge t+k+3$ with $s-t-k-1\equiv 0 \pmod{2}$ and define the integers $c(s,t,k)$, $g(s,t,k)$ and $d(s,t,k)$ in the following way. Set
$$ c(s,t,k) := \lfloor \frac{ \binom{s+3}{3} -sd_{t,k} - 6-3(s-t-k-1)/2+g_{t,k}}{s-1}\rfloor,$$
$$ d(s,t,k) := \binom{s+3}{3} -sd_{t,k} -6 -3(s-t-k-1)/2 +g_{t,k}-(s-1)c(s,t,k),$$ and $g(s,t,k):= c(s,t,k) - 3 -3(s-t-k-1)/2$. Note that
\begin{equation}\label{eq+a2}
s(d_{t,k} +c(s,t,k)) + 3-g_{t,k} -g(s,t,k) +d(s,t,k) = \binom{s+3}{3}, \ 0 \le d(s,t,k) \le s-2
\end{equation}
and (\ref{eq+a2}) holds even if $s=t+k+1$. From (\ref{eq+a2}) for the integers $s+2$ and $s$  
and the equality $g(s+2,t,k) -g(s,t,k) = c(s+2,t,k) -c(s,t,k) -3$ we get
\begin{align}\label{eq+a3}
& 2d_{t,k} + 2c(s,t,k) + (s+1)(c(s+2,t,k)-c(s,t,k)) + \notag\\
& d(s+2,t,k) -d(s,t,k) +3 = (s+3)^2
\end{align}

\begin{remark}\label{aa+0}
We have $c(2t+1,t,t) =t+3$, $d(2t+1,t,t)=0$, $c(2t,t,t-1) =t+2$, $d(2t,t,t-1) = t-1$, $c(2t+2,t,t-1)=2t+6$, $d(2t+2,t,t-1) = 2t-3$, 
$c(2t+3,t,t) = 2t+7$, $d(2t+3,t,t) =2t-1$.
\end{remark}
 
\begin{remark}\label{o+aa1}
We explain here the main reason for the assumption $t\ge 27$ made in this section. Fix an integer $s\ge t+k+1$ with $s\equiv t+k+1 \pmod{2}$. We work with a curve $X = C_{t,k}\sqcup A$ with $A$ a general smooth curve of degree $c(s,t,k)$ and genus $g(s,t,k)$ and we need $h^1(N_X(-2)) =0$, i.e. we need $h^1(N_A(-2)) =0$. We have $c(s,t,k) \ge g(s,t,k)+3$. By Lemma \ref{aa+1} below we have 
$g(s,t,k) \ge g(t+k+1,t,k)$. We have $g(2t+1,t,t)=t \ge 27$ and $g(2t,t,t-1) =t-1 \ge 26$. Since $g(s,t,k)\ge 26$, Remark \ref{ob2.1} gives $h^1(N_A(-2)) =0$.
\end{remark}

\begin{lemma}\label{r+00}
For all integer $s\ge t+k+1$ with $s\equiv t+k-1 \pmod{2}$ and $t\ge 27$ we have $g_{\lceil (s+1)/2\rceil, \lfloor (s+1)/2\rfloor} >g_{t,k} +g(s,t,k)$.
\end{lemma}

\begin{proof}
The lemma is true if $s =t+k+1$ by the explicit value of $g(t+k+1,t,k) =c(t+k+1,t,k)-3$ (Remark \ref{aa+0}).  
Now let $s\ge t+k+3$ and assume that the lemma is true for the integer $s-2$. Since $s-2 \ge t+k+1$ the inductive assumption
gives $g_{\lceil (s-1)/2\rceil, \lfloor (s-1)/2\rfloor}  > g_{t,k} +g(s-2,t,k)$. Thus it is sufficient to check that 
$g_{\lceil (s+1)/2\rceil, \lfloor (s+1)/2\rfloor}- g_{\lceil (s-1)/2\rceil, \lfloor (s-1)/2\rfloor} \ge g(s,t,k)-g(s-2,t,k)= c(s,t,k)-c(s-2,t,k) -3$. 
An elementary numerical computation shows that this inequality holds for any $s > t\ge 27$: indeed, the key point is that the difference 
$g_{\lceil (s+1)/2\rceil, \lfloor (s+1)/2\rfloor}- g_{\lceil (s-1)/2\rceil, \lfloor (s-1)/2\rfloor}$ is quadratic in $s$ by definition of $g_{t,k}$, while the difference $c(s,t,k)-c(s-2,t,k)$ is linear in $s$ by (\ref{eq+a3}).
\end{proof}

\begin{lemma}\label{aa+1}
For each $s\ge t+k+1$ with $s\equiv t+k-1 \pmod{2}$ we have $2(c(s+2,t,k)-c(s,t,k)) \ge s+4$.
\end{lemma}

\begin{proof}
Since $g_{t,k} +g(s,t,k) <g_{\lceil (s+1)/2\rceil, \lfloor (s+1)/2\rfloor}$ (Lemma \ref{r+00}), (\ref{eq+a2}) for $s, t, k$ and (\ref{eqi1}) for $t'=\lceil (s+1)/2\rceil$ and $k'=\lfloor (s+1)/2\rfloor$ imply $d_{t',k'} \ge c(s,t,k)+d_{t,k}$. Remark  \ref{aa+0} gives $c(s+2,t',k') =k'+3$. Since $0\le d(s+2,t,k) \le s$ and 
$0\le d(s,t,k) \le s-2$, (\ref{eq+a3}) and the difference between (\ref{eq+a2}) for $s':= s+2$ and (\ref{eq+bb+1}) for $t', k'$ imply 
$c(s+2,t,k)-c(s,t,k) \ge -1+c(s+2,t',k') =\lfloor (s+1)/2\rfloor +2$.
\end{proof}

Let $Q := \PP^1\times \PP^1$. The elements of $|\Oo _Q(0,1)|$ are the fibers of the projection $\pi _2:Q \to \PP^1$, 
so that each $D\in |\Oo _Q(1,0)|$ contains exactly one point of each fiber of $\pi _2$.

\quad {\bf{Assertion}} $M(s,t,k)$, $k\in \{t-1,t\}$, $s\ge t+k+1$, $s\equiv t+k+1 \pmod{2}$: Set $e=1$ if $0\le d(s,t,k) \le c(s+2,t,k)-c(s,t,k) -3$ and $e=2$ 
if $d(s,t,k) > c(s+2,t,k)-c(s,t,k)-3$. There is a $6$-tuple
 $(X,Q,D_1,D_2,S_1,S_2)$ such that
 \begin{itemize}
 \item[(a)] $Q$ is a smooth quadric surface, $X = C_{t,k}\sqcup Y$, $Y$ is a smooth curve of degree $c(s,t,k)$ and genus $g(s,t,k)$
 and $Q$ intersects transversally
 $X$, with no line of $Q$ containing $\ge 2$ points of $X\cap Q$;
 \item[(b)] $D_1,D_2$ are different elements of $|\Oo _Q(1,0)|$, each of them containing one point of $Y\cap Q$, $S_i\subset D_i\setminus D_i\cap Y$, $1\le i \le 2$,
 and $\sharp (S_1)+\sharp (S_2) =d(s,t,k)$; $\pi _2(S_2) \subseteq \pi _2(S_1)$;
$S_2=\emptyset$ and $\pi _2(S_1)\subseteq \pi _2(Y\cap (Q\setminus (D_1\cup D_2)))$ if $e=1$, $\sharp (S_2) = d(s,t,k) -c(s+2,t,k)+c(s,t,k) +3$  and
$\pi _2(S_2)\subseteq \pi _2(Y\cap (Q\setminus (D_1\cup D_2)))$ if $e=2$;
 \item[({c})] $h^i(\Ii _{X\cup S_1\cup S_2}(s)) =0$, $i=0,1$.
 \end{itemize} 
 
\begin{remark}\label{+a2.0}
Fix lines $L, R\subset \PP^3$ such that $L\cap R=\emptyset$ and $o\in \PP^3\setminus (L\cup R)$. Let $\ell : \PP^3\setminus \{o\} \to \PP^2$ denote the linear projection
from $o$. We have $\sharp (\ell (L)\cap \ell ({R}))=1$, i.e. there is a unique line $D(L,R,o)\subset \PP^3$ such that $o\in D(L,R,o)$, $D(L,R,o)\cap L\ne \emptyset$
and $D(L,R,o)\cap R\ne \emptyset$. We have $\sharp (D(L,R,o)\cap L) =\sharp (D(L,R,o)\cap R) =1$. The function $(L,R,o)\mapsto D(L,R,o)$ is regular.
\end{remark}

\begin{remark}\label{notation}
For any $o\in \PP^3$ let $\chi(o)$ denote the first infinitesimal neighbourhood of $o$ in $\PP^3$, i.e. the closed subscheme of $\PP^3$ with $\mathcal{I}_o^2$ as its ideal sheaf. 
For any surface $F \subset \PP^3$ and any scheme $B \subset \PP^3$ let $\Res_F(B)$ denote the closed subscheme of $\PP^3$ with $\mathcal{I}_B : \mathcal{I}_F$ as its 
ideal sheaf. We have $\Res_F(B) \subseteq B$. If $B$ is the disjoint union of closed subschemes $B_1$ and $B_2$ then $\Res_F(B)=\Res_F(B_1) \cup \Res_F(B_2)$. If $B$ is 
reduced then $\Res_F(B)$ is the union of the irreducible components of $B$ not contained in $F$. If $o \notin F$ then $\Res_F(\chi(o))=\chi(o)$. If $o \in F$ and $F$ is smooth 
at $o$ then $\Res_F(\chi(o))=\{o\}$.
\end{remark}

\begin{lemma}\label{+a1}
For all $t \ge 27$ and $k\in \{t-1,t\}$ assertion $M(t+k+1,t,k)$ is true.
\end{lemma}

\begin{proof}
Fix $C_{t,k}$ intersecting $Q$ at $2d_{t,k}$ general points (\cite{pe}).

\quad (a) Assume $k=t$. We have $c(2t+1,t,t) =t+3$ and $d(2t+1,t,t)=0$ and so we take $e=1$ with $S_1 =S_2 =\emptyset$. Take any $A\in |\Oo _Q(2,t+1)|$ with $A\cap C_{t,k}=\emptyset$. We have $\Res _Q(C_{t,t}\cup A) =C_{t,t}$ and thus $h^i(\Ii _{\Res _Q(C_{t,t}\cup A)}(2t-1)) =0$, $i=0,1$. We have $h^i(Q,\Ii _{Q\cap (C\cap A)}(2t+1,2t+1))
= h^i(Q,\Ii _{C_{t,t}\cap Q}(2t-1,t)) =0$, $i=0,1$, by (\ref{eq+a1}) and the generality of $C_{t,k}\cap Q$. Hence $h^i(\Ii _{C_{t,k}\cup A}(2t+1)) =0$, $i=0,1$.

We deform $A$ to a curve $Y$ of degree $t+3$ and genus $t$ with $Y\cap C_{t,k}=\emptyset$, $Y$ intersecting transversally $Q$ and with no line of $Q$ containing $\ge 2$
points of $Q\cap (C_{t,k}\cup Y)$. By the semicontinuity theorem for cohomology (\cite[III.8.8]{hart}), 
for a general $Y$ we have $h^i(\Ii _{C_{t,k}\cup Y}(2t+1)) =0$, $i=0,1$. Set $X:= C_{t,k}\cup Y$, $S_1=S_2=\emptyset$
and take as $D_1$ and $D_2$ any two different elements of $|\Oo _Q(1,0)|$, each of them containing one point of $Y\cap Q$.

\quad (b) Assume $k =t-1$. We have $c(2t,t,t-1) =t+2$, $d(2t,t,t-1) = t-1$ and $c(2t+2,t,t-1) -c(2t,t,t-1)=t+4$ (Remark \ref{aa+0}). Hence $e=1$.  However, in the proof of $M(t+k+1,t,k)$ we will exchange the two rulings (as we will do below for the general proof that $M(s,t,k)\Longrightarrow M(s+2,t,k)$), so that $D_1,D_2\in |\Oo _Q(0,1)|$. Take lines $L_1,L_2\in |\Oo _Q(1,0)|$
such that $L_1\ne L_2$ and $C_{t,t-1}\cap (L_1 \cup L_2) =\emptyset$, and $t$ different lines $R_j\in |\Oo _Q(0,1)|$, $1\le j\le t$, none of them containing
a point of $C_{t,t-1}\cap Q$. Fix $D_1,D_2\in   |\Oo _Q(0,1)|$ containing no point of $C_{t,t-1}\cap Q$ and with $D_h\ne R_j$ for all $h, j$. Set $u_h:= L_1\cap D_h$, $h=1,2$. 
Fix $E_1\subset D_1$ with $\sharp (E_1) =t-1$ and $E_1\cap (L_1\cup L_2) =\emptyset$. We have $h^1(Q,\Ii _{E_1}(2t-2,t)) =0$. Since $C_{t,k}\cap Q$ is a general
subset of $Q$ with cardinality $2d_{t,k}$, we
have $h^i(Q,\Ii _{Q\cap (C\cap A)\cup E_1}(2t,2t))
= h^i(Q,\Ii _{(C_{t,t}\cap Q)\cup E_1}(2t-2,t)) =0$, $i=0,1$, by (\ref{eq+a1}). The residual sequence of $Q$ gives $h^i(\Ii _{C_{t,k}\cup A\cup E_1}(2t)) =0$, $i=0,1$.

Take an ordering $\{o_1,\dots ,o_{t-1}\}$ of $E_1$ and let $M_i$ the only
element of $|\Oo _Q(1,0)|$ with $o_i\in M_i$. Set $w_i:= R_i\cap M_i$, $1\le i\le t-1$. We fix a deformation $\{L_h(\lambda )\}_{\lambda \in \Lambda}$, $h=1,2$, of $L_h$ with the following properties: $\Lambda$ is a connected and affine smooth curve, $o\in \Lambda$, $L_h(o) =L_h$, $u_h\in L_h(\lambda)$ for all $\lambda$, $L_1(\lambda )\cap L_2(\lambda) =\emptyset$ for all $\lambda$ and $L_h(\lambda )$ is transversal to $Q$ for all $\lambda \ne o$. For each $i$ with $1\le i\le t-1$ there is a unique line $R_i(\lambda)$
containing $w_i$ and intersecting both $L_1(\lambda )$ and $L_2(\lambda)$ (Remark \ref{+a2.0}). There is a deformation $\{R_t(\lambda)\}_{\lambda \in \Lambda}$ of $R_t$ with $R_t(o) =R_t$, $R_t(\lambda )$
intersecting both $L_1(\lambda )$ and $L_2(\lambda )$. Taking instead of $\Lambda$ a smaller neighborhood of $o$ we may assume $R_i(\lambda )\cap R_j(\lambda )=\emptyset $ for all $i\ne j$ and all $\lambda$ so that
 $A(\lambda ):= L_1(\lambda)\cup L_2(\lambda)\cup R_1(\lambda )\cup \cdots \cup R_t(\lambda)$ is a connected nodal curve of degree $t+2$ and arithmetic genus $t-1$. By semicontinuity (restricting if necessary $\Lambda$ to a neighborhood of $o$) we have $h^i(\Ii _{C_{t,k}\cup A(\lambda )\cup E_1}(2t)) =0$, $i=0,1$, for all $\lambda \in \Lambda$.
 Fix $\lambda _0\in \Lambda \setminus \{o\}$. Let $\{B_\delta\}_{\delta \in \Delta}$ be a smoothing of $A(\lambda _0)$ fixing $u_1$ and $u_2$, i.e. take a smooth and connected affine curve $\Delta$
 and $a\in \Delta$ with $B_a = A(\lambda _0)$, $B_\delta$ a smooth curve of degree $t+2$ and genus $t-1$ and $\{u_1,u_2\}\subset B_\delta$ for all $\delta$. Restricting if necessary $\Delta$ we may assume
 that $B_\delta$ is transversal to $Q$ and disjoint from $C_{t,k}\cup E_1$ for all $\delta \in \Delta$ and (by semicontinuity) that  $h^i(\Ii _{C_{t,k}\cup B_\delta \cup E_1}(2t)) =0$, $i=0,1$. Since $A(\lambda _0)$ is transversal to $Q$, we may (up to a finite covering of $\Delta$) find $t-1$ sections $s_1,\dots ,s_{t-1}$ of the family $\{B_\delta \cap Q\}_{\delta \in \Delta}$ of $2t+4$ ordered points of $Q$ with
 $s_i(a) = w_i$, $i=1,\dots ,t-1$. Let $M_j(\delta)$, $\delta \in \Delta$, be the  only
element of $|\Oo _Q(1,0)|$ with $w_i\in M_i(\delta )$. Set $o_i(\delta ):= L_1\cap M_i(\delta )$ and $E_1(\delta ):= \{o_1(\delta ),\dots ,o_{t-1}(\delta )\}$. By semicontinuity for a general $\delta \in \Delta \setminus \{a\}$ we have $h^i(\Ii _{C_{t,k}\cup B_\delta \cup E_1(\delta )}(2t)) =0$. We fix such a $\delta$ and set
$X:= C_{t,k}\cup B_\delta$, $S_1:= E_1(\delta )$, $S_2:=\emptyset$. For $M(2t,t,t-1)$ we use the lines $D_1$, $D_2$ and $M_j(\delta )$, $1\le j\le t-1$.
\end{proof}

\begin{lemma}\label{+na1}
For each integer $s\ge t+k+1$ such that $s\equiv t+k+1 \pmod{2}$ we have $2c(s,t,k) \ge s+4$ and $2c(s,t,k) \ge s+6$ is $s\ge t+k+3$.
\end{lemma}

\begin{proof}
The case $s =t+k+1$ is true by Remark \ref{aa+0}. The general case follows by induction $s-2\Longrightarrow s$ by Lemma \ref{aa+1}.
\end{proof}

We need the following auxiliary result, proved in \cite[Lemma 2.5]{ccg} and \cite[bottom of page 176]{hh0}.

\begin{lemma}\label{t1.1}
Fix lines $D, L\subset \PP^3$ such that $D\cap L$ is a point $o$ and $q\in L\setminus \{o\}$. Then there is a family $\{L_\lambda\}_{\lambda\in \KK}$ of lines of $\PP^3$
such that $L_0 = L$, $L_\lambda \cap D =\emptyset$ for all $t\ne 0$, $D\cup L\cup \chi (o)$ is a flat limit of the family $\{D\cup L_\lambda\}_{\lambda \in \KK \setminus \{0\}}$.
\end{lemma}

\begin{proof}
Take homogenous coordinates $x_0,x_1,x_2,x_3$ such that $o = (1:0:0:0)$, $D =\{x_1=x_2=0\}$, $L = \{x_1=x_3=0\}$ and $q = (0:0:1:0)$. Take
$L_\lambda = \{x_1+\lambda x_0=x_3=0\}$. Note that $L_0=L$ and that $L \cap D =\emptyset$ for all $\lambda\ne 0$. Set $Y_\lambda:= D\cup L_\lambda$. For $\lambda \ne 0$ the ideal sheaf of scheme $Y_t$
is generated by the quadrics $x_1(x_1+\lambda x_0)$, $x_2(x_1+\lambda x_0)$, $x_1x_3$, $x_2x_3$, while the ideal sheaf of the scheme $Y_0$ is determined by the quadrics $x_1^2$, $x_1x_2$, $x_1x_3$ and $x_2x_3$. This algebraic family of projective schemes is flat because it has constant Hilbert polynomial \cite[III.9.8.4]{hart}.
\end{proof}

\begin{lemma}\label{+a2}
Assume $t\ge 27$ and $k\in \{t-1,t\}$. Fix an integer $s\ge t+k+1$ such that $s\equiv t+k+1 \pmod{2}$. If $M(s,t,k)$ is true, then $M(s+2,t,k)$ is true.
\end{lemma}

\begin{proof}
Let $e\in \{1,2\}$ be the integer arising in $M(s,t,k)$ and $f\in \{1,2\}$ the corresponding integer for $M(s+2,t,k)$.  Take $(X,Q,D_1,D_2,S_1,S_2)$ satisfying $M(s,t,k)$ with $X = C_{t,k}\sqcup Y$ and $D_1,D_2\in |\Oo _Q(1,0)|$. The $6$-tuple $(X',Q,D'_1,D'_2,S'_1,S'_2)$ will be a solution after exchanging the two rulings of $Q$, i.e. we will take $D'_1,D'_2\in |\Oo _Q(0,1)|$
and we use $\pi _1$ instead of $\pi _2$. In each step with $d(s,t,k)\ne 0$ we obtain $X'$ smoothing a curve $W$ union of $X$, $\chi:= \cup _{o\in S_1\cup S_2} \chi (o)$, $e+1$ elements $|\Oo _Q(1,0)|$ and $c(s+2,t,k)-c(s,t,k)-e-1$ elements of $|\Oo _Q(0,1)|$. See step ({c}) for the easier case $d(s,t,k)=0$ (here to get $W$ we add to $X$ a line $D_0\in |\Oo _Q(1,0)|$ and $c(s+2,t,k)-c(s,t,k)-1$  elements of $|\Oo _Q(0,1)|$).

\quad (a) Assume $e=2$ and set $z:= d(s,t,k) +3 -c(s+2,t,k) +c(s,t,k)$. Since $d(s,t,k) \le s-2$, Lemma \ref{aa+1} gives $d(s,t,k) \le 2(c(s+2,t,k) -c(s,t,k)-3)$, i.e.
$z \le c(s+2,t,k) -c(s,t,k)-3$. By assumption there is $E\subset Y\cap (Q\setminus (D_1\cup D_2))$ such that $\sharp (E)=z$ and $\pi_2(E)=
\pi _2(S_2) \subseteq \pi _2(S_1)$. Take a line $D_0\in |\Oo _Q(1,0)|$ different from $D_1,D_2$,  with $D_0\cap E =\emptyset$, $D_0\cap C_{t,k}\cap Q =\emptyset$ and $D_0\cap Y\cap Q\ne \emptyset$; we use
that $2c(s,t,k) \ge 3+z$ (Lemma \ref{+na1}).  Take distinct lines $L_i\in |\Oo _Q(0,1)|$, $1\le i \le c(s+2,t,k) -c(s,t,k)-3$, such that $L_i\cap Y \ne \emptyset$
 if and only if $i\le z$, $X\cap (\bigcup _{i=1}^{c(s+2,t,k)-c(s,t,k) -3} L_i )=E$, $L_i\cap (C_{t,k}\cap Q) =\emptyset$ for all $i$. Set $J:= (D_0\cup D_1\cup D_2) \cup (\bigcup _{i=1}^{c(s+2,t,k)-c(s,t,k) -3} L_i)$. We fix $f$ general lines $R_i\in |\Oo _Q(0,1)|$, $1\le i \le f$, and $A_i\subset R_i$, $1\le i\le f$, with the
 conditions $\sum _{i=1}^{f} \sharp (A_i) =b(s+2,t,k)$,  $\pi _1(A_f)\subseteq \pi _1(A_1) $ and $\pi _1(A_f)\subseteq \pi _1(Y\cap (Q\setminus J))$. Set $\chi := \cup _{o\in S_1\cup S_2} \chi (o)$, $A:= A_1\cup A_2$ and $W:= X\cup J \cup \chi$. 
 
 \quad \emph{Claim 1:} There is an affine connected smooth curve $\Delta$ such that the scheme $Y\cup J\cup \chi$ is a flat degeneration of a family of unions of $Y\cup D_0\cup D_1\cup D_2$
 and $c(s+2,t,k) -c(s,t,k)-3$ lines $L_{i\lambda}$, $\lambda \in \Delta$,  such that $L_{i\lambda} \cap Y \ne \emptyset$
 if and only if $i\le z$, $D_0\cap L_{i\lambda} =D_0\cap L_i$ for all $i$, $D_1\cap L_{i\lambda}  =\emptyset$ for all $\lambda \in \Delta$ and all $i$, and $D_2\cap L_{i\lambda}\ne \emptyset$ (and it is a point) if and only if $z+1 \le i \le c(s+2,t,k)-c(s,t,k)-3$.  
 
 \quad \emph{Proof of Claim 1:} As $\Delta$ we take a suitable curve contained in the affine manifold $\prod _{i=1}^{c(s+2,t,k) -c(s,t,k)-3} \Delta _i$, which we are now going to define 
(in the use of Claim 1 we only need that $\Delta$ is irreducible, so we could use a similar claim, but with this connected manifold instead of the irreducible curve $\Delta$). 
Each $\Delta _i$ is a smooth connected curve, so $\prod _{i=1}^{c(s+2,t,k) -c(s,t,k)-3} \Delta _i$ is irreducible.
We describe each $L_{i\lambda}$ with $L_i$ as a limit separately for each $i$; $\Delta _i$ is the parameter space for the line $L_{i\lambda}$. We need to modify the proof of Lemma \ref{t1.1} in the following way. First assume $z< i \le c(s+2,t,k)  -c(s,t,k)-3$. In this case we fix the point $q_i:= D_0\cap L_i$ and use as $\Delta _i$ a Zariski open neighborhood of $D_2\cap L_i$; for each $q\in D_2$ there is a unique line $L(q_i,q)$ containing $\{q_i,q\}$; when $q$ goes
to $D_2\cap L_i$ the line $L(q_i,q)$ goes to the line; we need to restrict $\Delta _i$ to avoid the points $q$ such that $L(q_i,q)\cap (Y\cup C_{t,k} \cup D_1)\ne \emptyset$. Now assume
$i\le z$. We fix the point $q_i:= Y\cap L_i$ and take as $\Delta _i$ a Zariski neighborhood of $q_i$ in $Y$; since $q_i\notin D_0$ for each $q\in D_0\setminus Y\cap D_0$ there
is a unique line $L(q_i,q)$ containing $\{q_i,q\}$; we need to restrict $\Delta _i$  to avoid the points $q$ such that $L(q_i,q)\cap (Y\cup C_{t,k})\ne \{q\}$. We restrict $\Delta _1\times
\cdots \times \Delta _{c(s+2,t,k) -c(s,t,k)-3}$ to a non-empty Zariski open subset $U$ such that for all $\lambda \in U$ the union $U'$ of $Y\cup D_0\cup D_1\cup D_2$
and the line $L_{i\lambda}$, $1\le i \le c(s+2,t,k) -c(s,t,k)-3$ is nodal and it has no singular point which is not prescribed by the construction. Note that $U'$ is connected
and $p_a(U') = g(s+2,t,k)$.
 
 \quad \emph{Claim 2:} $W$ is a flat degeneration of a disjoint union of $C_{t,k}$ and a smooth curve of degree $c(s+2,t,k)$ and genus $g(s+2,t,k)$.
 
 \quad \emph{Proof of Claim 2:} Since $C_{t,k}\cap (J\cap \chi) =C_{t,k}\cap Y = \emptyset$, it is sufficient to prove that $Y\cup J\cup \chi$ is a flat degeneration of a family of smooth curves of degree  $c(s+2,t,k)$ and genus $g(s+2,t,k)$. By Claim 1, $Y\cup J\cup \chi$ is a flat degeneration of a family of unions of $Y\cup D_0\cup D_1\cup D_2$ and $c(s+2,t,k)
 -c(s,t,k)-3$ disjoint lines, none of them intersecting $D_1\cup D_2$ and each of them intersecting $Y\cup D_0\cup D_1\cup D_2$ quasi transversely at exactly two points. 
We first prove that $Y\cup J\cup \chi$ is a flat degeneration of a family of unions of $Y\cup D_0\cup D_1\cup D_2$
 and $c(s+2,t,k) -c(s,t,k)-3$ lines $L_{i\lambda}$, $\lambda \in \KK\setminus \{0\}$,  such that $L_{\cap i\lambda} \cap Y \ne \emptyset$
 if and only if $i\le z$, $D_0\cap L_{i\lambda} =D_0\cap L_i$ for all $i$ and $D_1\cap L_{i\lambda} = D_2\cap L_{i\lambda} =\emptyset$ for all $\lambda \in \KK\setminus \{0\}$.
We may do this smoothing separately, first for $D_0, D_1,D_2$ and then for each line $L_{i\lambda}$ quoting each time Lemma \ref{aaa1} and following the deformation with a family of lines with special fiber
$L_{j\lambda}$, $j\ne i$, because any two points of $\PP^3$ uniquely determine a line and the line depends regularly if we move regularly the two points (see the Side Remark below), but we may do all
the smoothings simultaneously just choosing the appropriate references from \cite{hh} or \cite{s} or other sources, e.g. Lemma \ref{aaa1} or \cite[Corollary 5.2]{hh}.

\quad \emph{Side Remark:} As the reader may have noticed in the proof Claim 1 we only used Lemma \ref{t1.1} (i.e. a known result) and the fact that two different points $q, q'$
of $\PP^r$, $r\ge 3$, uniquely determine a line $L(q,q')\subset \PP^3$ and the regularity of the map $(q,q') \to L(q,q')$ from $\PP^r\times \PP ^r\setminus \Delta_{\PP^r}$, where $\Delta _{\PP^r}$ is the diagonal,  to the Grassmaniann $G(1,r)$. Lemma \ref{t1.1} is true in a more general situation, as the flat limit of a two smooth germs of curves
colliding to an ordinary node; call $o$ this nodal point. Instead of the lines $D, L$ we take the tangent lines to the two smooth germs of curves. Their linear span determines an element of $G(3,r)$ and we consider the first infinitesimal neighborhood $\chi (o)$ of $o$ in a $3$-dimensional projective space which is a limit of these elements of $G(3,r)$. In the literature Claims 1 and 2 are often used, but without separating them. For the algebraic geometers of our generation the first instance of this flat limit with a nilpotent was \cite[III.9.8.4 and figure 11 at page 260]{hart}.
 
To obtain a smoothing of $W$ as in the Claim 2, but compatible with the data $A_1,A_2$, see steps (a1) and (a2). 
We have $\Res _Q(W\cup A) = X\cup S_1\cup S_2$ and so $h^i(\Ii _{\Res_Q(W\cup A)}(s)) =0$, $i=0,1$. We have
$h^i(Q,\Ii _{(W\cap Q)\cup A}(s+2,s+2)) = h^i(Q,\Ii _{(X\cap (Q\setminus J)\cup A}(s-1,s+5+c(s,t,k)-c(s+2,t,k)))$. We have $\sharp ((X\cap (Q\setminus J))\cup A)=h^0(Q,\Oo_Q(s-1,s+5+c(s,t,k)-c(s+2,t,k))$. We have $h^1(Q,\Ii _A(s-1,s+5+c(s,t,k)-c(s+2,t,k))) =0$, because $s+5+c(s,t,k)-c(s+2,t,k)>0$, $f\le 2$ and $\sharp (A_1)\le s$; this is a key reason for our definition of $M(s+2,t,k)$.
Therefore to prove that $h^i(Q,\Ii _{(X\cap (Q\setminus J)\cup A}(s-1,s+5+c(s,t,k)-c(s+2,t,k)))=0$, $i=0,1$, it is sufficient to prove that we may take as $X\cap (Q\setminus J)$
a general subset of $Q$ with its prescribed cardinality. By Remark \ref{o+aa1} we have $h^1(N_X(-2)) =0$. Since $h^1(N_X(-2)) =0$, we may deform $X$ keeping fixed $E$
so that the other points are general in $Q$. 

\quad (a1) We have just proved that $h^i(\Ii _{W\cup A}(s+2)) =0$, $i=0,1$. If $d(s+2,t,k) =0$, then $M(s+2,t,k)$ is proved for $e=2$.
Now assume $d(s+2,t,k)>0$. To prove $M(s+2,t,k)$ when $e=2$ we need to deform $W$ to a smooth $X'=C_{t,k}\sqcup Y'$ intersecting
transversally $Q$ and (perhaps moving $A$) to obtain condition (b) of $M(s+2,t,k)$. Set $P_i:= Y\cap D_i$, $i=0,1,2$. Let $\{D_i(\lambda )\}_{\lambda \in \Lambda}$ be a deformation of $D_i$ with
$\Lambda$ a smooth and connected affine curve, $o\in \Lambda$, $D_i(o) =D_i$, $D_i(\lambda )$, $\lambda \in \Lambda \setminus \{o\}$, a line of $\PP^3$ transversal
to $Q$ and containing $P_i$. Fix $i\in \{1,\dots ,z\}$. By Remark \ref{+a2.0} for each $\lambda \in \Lambda$ there is a unique line $L_i(\lambda )\subset \PP^3$
such that $D_0\cap L_i\in L_i(\lambda)$, $L_i(\lambda )\cap D_1(\lambda )\ne \emptyset$ and $L_i(\lambda )\cap D_2(\lambda )\ne \emptyset$; restricting if necessary $\Lambda$ we
may assume that all $L_i(\lambda)$, $\lambda \ne o$, are transversal to $Q$.
Fix an integer $i$ with $z<i\le c(s+2,t,k)-c(s,t,k)-3$ and fix a general $m_i\in L_i$. By Remark \ref{+a2.0} there is a unique line $L_i(\lambda)$ such that $m_i\in L_i(\lambda)$, $L_i(\lambda )\cap D_1(\lambda )\ne \emptyset$ and $L_i(\lambda )\cap D_2(\lambda )\ne \emptyset$; restricting if necessary $\Lambda$ we
may assume that all $L_i(\lambda)$, $\lambda \ne o$, are transversal to $Q$. Restricting if necessary $\Lambda$ to a smaller neighborhood
of $o$ in $\Lambda$ we may assume that $L_i(\lambda )\cap L_j(\lambda) =\emptyset$ for all $i\ne j$, that $C_{t,k}\cap L_i(\lambda ) =\emptyset$ for all $i$ and all $\lambda$,
that $L_i(\lambda )\cap D_0\ne \emptyset$ if and only if $i\le z$. Fix a general $\lambda \in \Lambda$ and set $J(\lambda):= D_0(\lambda)\cup D_1(\lambda) \cup D_2(\lambda) \cup (\bigcup _{i=1}^{c(s+2,t,k)-c(s,t,k)-3} L_i(\lambda))$. Let $\chi (\lambda)$ be the union of all $\chi (q)$
with either $q\in D_1(\lambda )\cap L_i(\lambda)$, $1\le i \le c(s+2,t,k)-c(s,t,k) -3$ or $q\in D_2(\lambda )\cap L_i(\lambda)$, $1 \le i\le z$. Set $W(\lambda ):=X\cup J(\lambda)\cup \chi (\lambda)$. $W(\lambda)$ is the disjoint union of $C_{t,k}$ and of a degeneration
of a flat family of smooth and connected curves of degree $c(s+2,t,k)$ and genus $g(s+2,t,k)$.
As in the first part of step (a), restricting if necessary $\Lambda$, by semicontinuity we get $h^i(\Ii _{W(\lambda )\cup A}(s+2)) =0$, $i=0,1$.

\quad (a2) To prove $M(s+2,t,k)$ we need to prove that there is a set like $A$ (call it $A'$) satisfying both $h^i(\Ii _{W(\lambda )\cup A'}(s+2)) =0$, $i=0,1$, and condition (b)
of $M(s+2,t,k)$. First of all, instead of $P_i$, $0\le i \le 2$, we take a family $\{P_i(\lambda )\}_{\lambda \in \Lambda}$ of points of $Y$ with $P_i(o) =P_i$
and $P_i(\lambda )\in Y\setminus Y\cap Q$ for all $\lambda \in \Lambda \setminus \{o\}$. Assume for the moment $f=2$. We modify the definition of $D_i(\lambda )$, because we 
impose that $P_i(\lambda )\in D_i(\lambda)$ (instead of $P_i\in D_i$), but we also impose that $D_1(\lambda )\cap R_1\ne \emptyset$
and $D_2(\lambda )\cap R_2\ne \emptyset$ (this is possible by Remark \ref{+a2.0}). Then we construct $L_i(\lambda )$ as above. With this new definition
$R_1$ and $R_2$ are secant lines of $W(\lambda )\setminus (C_{t,k}\cup Y)$ , $Y\subset W(\lambda)$, $\pi _1(A_2)\subseteq \pi _1(A_1)$ and $\pi _1(A_f)\subseteq \pi _1(Q\cap (Y\setminus J(\lambda )\cap Y))$; call $m_1,\dots ,m_x$, $x=\sharp (A_f)$, the points of $Y\cap Q$ whose image is $\pi _1(A_f)$. We fix $\lambda\in \Lambda \setminus \{o\}$. Let $\{B_\delta \}_{\delta \in \Delta}$ be a smoothing of $W(\lambda)$ with $\Delta$ an affine and connected smooth curve, $a\in \Delta$, and $B_a = W(\lambda)$. Set $A(a):= A$. Since $Y$ is transversal to $Q$,
up to a finite covering of $\Delta$ we may find $x+2$ sections $s_1,\dots ,s_x, z_1,z_2$ of the total space of $\{B_\delta \}_{\delta \in \Delta}$ with $s_i(a) =m_i$, $z_1(a) = R_1\cap D_1(\lambda)$, $z_2(a) = R_2\cap D_2(\lambda )$,
$s_i(\delta )\in B_\delta \cap Q$, $z_1(\delta )\in B_\delta \cap Q$ and $z_2(\delta )\in B_\delta \cap Q$
 for all $\Delta$. Let $R_h(\delta )$, $h=1,2$, be the only element of $|\Oo _Q(0,1)|$ containing $z_h(\delta )$. For each $\delta \in \Delta \setminus \{a\}$ and $i\in \{1,\dots ,x\}$ let $M_i(\delta)\in |\Oo _Q(1,0)|$ be the only line of this ruling
of $Q$ containing $s_i(\delta)$. Set $A_1(\delta ):= \cup _{i=1}^{x} (R_1(\delta)\cap M_i(\delta))$ and $A_2(\delta ):= \cup _{i=1}^{d(s+2,t,k)-x} (R_2(\delta )\cap M_i(\delta))$. Set $X_\delta := C_{t,k}\cup B_\delta$.
By construction $(X_\delta ,Q,R_1,R_2,A_1(\delta ),A_2(\delta ))$ satisfies condition (b) of $M(s+2,t,k)$, exchanging the two rulings of $Q$. By semicontinuity 
we have $h^i(\Ii _{B_\delta \cup A(\delta )}(s+2)) =0$, $i=0,1$, for a general $\delta \in \Delta$.

Now assume $f=1$. In this case we only impose that $D_i(\lambda )$ meets $R_1$; we have $\pi _1(A_1)\subset \pi _1(Q\cap (Y\setminus J(\lambda )\cap Y))$ and $x = \sharp (A_1)=
b(s+2,t,k)$.

\quad (b) Assume $e=1$ and $d(s,t,k)>0$, i.e. assume $0<d(s,t,k) \le c(s+2,t,k)-c(s,t,k)-3$. We set $S_2:=0$ and ignore $D_2$.  We fix $o\in S_1$. Take a line $D_0\ne D_1$ meeting
$Y\cap Q$ and $c(s+2,t,k)-c(s,t,k)-2$ distinct lines $L_i\in |\Oo _Q(0,1)|$, with $L_i\cap (C_{t,k}\cap Q) =\emptyset$ for all $i$, $L_i\cap (Y\cap Q)\ne \emptyset$
if and only if $1\le i \le d(s,t,k)-1$ and $S_1\setminus \{o\} =D_1\cap (L_1\cup \cdots \cup L_{d(s,t,k)-1})$. Set $J:= (D_0\cup D_1)\cup (\bigcup _{i=1}^{c(s+2,t,k)+c(s,t,k)-2}L_i)$ and $\chi := \cup_{o\in S_1} \chi (o)$. Note that $\chi (X\cup J\cup \chi) -\chi (X) = c(s,t,k)-c(s+2,t,k)+3$. To modify step (a2) we impose that $D_1(\lambda )\cap R_1\ne \emptyset$
and $D_0(\lambda )\cap R_2\ne \emptyset$.

\quad ({c}) Assume $d(s,t,k) =0$. Hence $S_1=S_2=\emptyset$. Take a line 
$D_0\in |\Oo _Q(1,0)|$ different from $D_1,D_2$ and with $D_0\cap Y\cap Q\ne \emptyset$. Take
$c(s+2,t,k)-c(s,t,k)-1$ lines $L_i\in |\Oo _Q(0,1)|$, $1\le i \le c(s+2,t,k)-c(s,t,k)-1$, such that $L_i\cap (C_{t,k}\cap Q)  =\emptyset$ for all $i$ and $L_i\cap (Y\cap Q)\ne \emptyset$
if and only if $1\le i \le c(s+2,t,k)-c(s,t,k)-3$. Set $J:= D_0\cup (\bigcup _{i=1}^{c(s+2,t,k)-c(m,t,k) -1} L_i)$, $Y':= Y\cup J$ and $W:= X\cup J$.  Note that $\chi (W) -\chi (X) = c(s,t,k)-c(s+2,t,k)+3$. The union $Y'$ is a connected nodal curve,
which is a flat degeneration of a family of smooth curves of degree $c(s+2,t,k)$ and genus $g(s+2,t,k)$ not intersecting $C_{t,k}$. As in step (a) we
get $h^1(\Ii _W(s+2)) =0$ and $h^0(\Ii _W(s+2)) =d(s+2,t,k)$. If $d(s+2,t,k) =0$, then we are done, because $A=\emptyset$ and so condition (b) of $M(s+2,t,k)$ is trivially true. Now assume $d(s+2,t,k) >0$. 

First assume $f=2$.  As in step (a) we prove $M(s+2,t,k)$ interchanging the rulings of $Q$ and set $x:= c(s+4,t,k)-c(s+2,t,k)-3$.
We fix general lines $R_1,R_2\in |\Oo _Q(0,1)|$ and take $A_i\subset R_i$ such that  $\pi _1(A_2)\subseteq \pi _1(A_1)\cap \pi _1(Q\cap (Y\setminus J\cap Y))$.
Set $A:= A_1\cup A_2$. For a general $X$ we have $h^i(\Ii _{W\cup A}(s+2)) =0$, $i=0,1$. Set $q:= D_0\cap Y$. By Remark \ref{+a2.0} there is a family
$\{D_0(\lambda )\}_{\lambda \in \Lambda}$ of lines of $\PP^3$ and $o\in \Lambda$ with $D_0(o) =D_0$, $\sharp (D_0(\lambda )\cap Y)=1$ for all $\lambda$,
$D_0(\lambda )\cap Y\notin Q$ if $\lambda \ne 0$, $D_0(\lambda )\cap R_1\ne \emptyset$ and $D_0(\lambda )\cap R_2\ne \emptyset$. Up to a finite covering
of $\Lambda$ we may also find families $\{L_i(\lambda )\}_{\lambda \in \Lambda}$, $1\le i\le c(s+2,t,k)-c(s,t,k)-1$. Set $J(\lambda) =D_0(\lambda ):= D_0\cup (\bigcup
_{i=1}^{c(s+2,t,k)-c(s,t,k)-1} L_i(\lambda ))$.
We do the smoothing of $Y\cup J(\lambda)$ as in step (a2).

Finally, if $f=1$ we only need $D_0(\lambda )\cap R_1\ne \emptyset$ for all $\lambda$.\end{proof}


\section{With a constant genus $g$}\label{horace}

We fix an integer $t\ge 27$ and take $k\in \{t-1,t\}$. We fix an integer $g\ge g_{t,k} +g(t+k+5,t,k)$. Let $y$ be the maximal integer $\ge t+k+5$ such that
$y\equiv t+k-1 \pmod{2}$ and $g_{t,k} +g(y,t,k) \le g$ ($y$ exists, because $\lim _{u\to+\infty} g(t+k+1+2u,t,k) =+\infty$). By the definition of $y$ we have $y\ge t+k+5$ and
$y \equiv t+k-1\pmod{2}$.  For all integers $x \ge y+2$ with $x\equiv y\pmod{2}$ define the integers $a(x,t,k,y)$ and $b(x,t,k,y)$ by the relation
\begin{equation}\label{eqd2}
xd_{t,k}+3-g  +xa(x,t,k,y) +b(x,t,k,y) = \binom{x+3}{3}, \ 0 \le b(x,t,k,y) \le x-1
\end{equation}

If $x\ge y+4$, by taking the difference between equation (\ref{eqd2}) and the same equation for the integer $x':= x-2$ we get
\begin{eqnarray}\label{eqd2.1+}
2d_{t,k} + 2a(x,t,k,y) + (x+2)(a(x+2,t,k,y)-a(x,t,k,y)) \notag\\
+b(x+2,t,k,y) -b(x,t,k,y) = (x+3)^2 
\end{eqnarray}

\begin{lemma}\label{+n2}
For each $x\ge y+2$ with $x\equiv y\pmod{2}$ we have $2(a(x+2,t,k,y)-a(x,t,k,y))\ge x+5$.
\end{lemma}

\begin{proof}
Assume by contradiction $2(a(x+2,t,k,y)-a(x,t,k,y))\le x+4$. Recall that for all $u\ge v>0$ we have
\begin{equation}\label{eq+d1}
(u+v-1)d_{u,v} +2-g_{u,v} =\binom{u+v+2}{3}
\end{equation}

\emph{Claim 1:} We have $g_{\lceil (y+3)/2\rceil,\lfloor (y+3)/2\rfloor} > g$.

\emph{Proof of Claim 1:} By the definition of $y$ we have $g(y+2,t,k) +g_{t,k} > g$. Thus to prove Claim 1 it is sufficient to use that $g(y+2,t,k) + g_{t,k} \le g_{\lceil (y+3)/2\rceil ,\lfloor (y+3)/2\rfloor}$ (Lemma \ref{r+00}).

First assume $x$ odd, i.e. $k=t$. Since $g_{(x+1)/2,(x+1)/2} \ge g_{(y+3)/2,(y+3)/2} >g$ by Claim 1, (\ref{eq+d1}) and (\ref{eqd2}) give $d_{(x+1)/2,(x+1)/2} \ge d_{t,k}+a(x,t,k,y)$. Since $b(x+2,t,k,y)\le x+1$
and $b(x,t,k,y)\ge 0$  (\ref{eqd2.1+})
gives
$$
(x+1)(x+3)/2 + (x+2)(x+4)/2 +x+1 \ge (x+3)^2,
$$
which is false. Now assume $x$ even, i.e. $k=t-1$. Since $g_{(x+2)/2,x/2}\ge g_{(y+4)/2,(y+2)/2} >g$ by Claim 1,  (\ref{eq+d1}) and (\ref{eqd2}) gives $d_{(x+2)/2,x/2} \ge d_{t,k}+a(x,t,k,y)$.
 Since $b(x+2,t,k,y)\le x+1$
and $b(x,t,k,y)\ge 0$  (\ref{eqd2.1+})
gives
$$
(x+2)^2/2 + (x+2)(x+4)/2 +x+1 \ge (x+3)^2,
$$
which is false.
\end{proof}

\begin{lemma}\label{+n3}
 We have $2(a(y+2,t,k,y) -c(y,t,k)) \ge y+5$.
 \end{lemma}
 
\begin{proof}
Define the integers $w, z$ by the relations
\begin{equation}\label{eq++1}
(y+2)(w+d_{t,k}) +3-g_{t,k} -g(y,t,k) +z =\binom{y+5}{3}, \ 0\le z\le y+1
\end{equation}
Since $g\ge g_{t,k} +g(y,t,k)$, we have $w\le a(y+2,t,k)$. Hence it is sufficient to prove that $2(w-c(y,t,k)) \ge y+5$. Taking the difference between (\ref{eq++1}) and the case $s=y$ of (\ref{eq+a2}) we get
$$2d_{t,k} +2c(y,t,k) +(y+2)(w-c(y,t,k)) +z-d(y,t,k) =(y+3)^2$$Then we continue as in the proof of Lemma \ref{+n2} with $y+2$ instead of $x+2$.\end{proof}

The next lemma follows at once by induction on $x$, the inequality $2c(y,t,k) \ge y+6$ and  Lemmas \ref{+n2} and \ref{+n3}.

\begin{lemma}\label{+n2.0}
We have $2a(x,t,k,y)\ge x+6$ for all integers $x\ge y+2$ with $x\equiv y\pmod{2}$.
\end{lemma}

\begin{lemma}\label{o+n2.0}
For each $x\ge y+2$ with $x\equiv y\pmod{2}$ we have $a(x,t,k,y) \ge g-g_{t,k}+3$.
\end{lemma}

\begin{proof}
First assume $x=y+2$. By the definition of the integer $y$ we have $g_{t,k} +g(y,t,k) \le g \le \tau := g_{t,k}+g(y+2,t,k)-1$. The integers $a(y+2,t,k,y)$ and $b(y+2,t,k,y)$ depend on the choice of $g$ and (only for this proof) we call them $a(y+2,t,k,y)_g$ and $a(y+2,t,k,y)_g$. Fix integers $q, q'$ such that $g_{t,k} +g(y,t,k) \le q\le q' \le \tau$. From (\ref{eqd2}) or (\ref{eqd2.1+}) for $q$ and $q'$ we get
\begin{align}\label{eqr+00.1}
& (y+2)(a(y+2,t,k,y)_{q'} -a(y+2,t,k,y)_q) = \notag \\
& b(y+2,t,k,y)_q -b(y+2,t,k,y) _{q'} +q'-q 
\end{align}
Since $0 \le b(y+2,t,k,y)_q \le y+1$ and  $0 \le b(y+2,t,k,y)_{q'} \le y+1$, (\ref{eqr+00.1}) implies $a(y+2,t,k,y) _q \le a(y+2,t,k,y) _{q'} \le a(y+2,t,k,y) _q +q'-q$. Thus to prove the lemma for $x=y+2$ it is sufficient to prove it for the genus $\tau$. We have, by (\ref{eq+a2}) and (\ref{eqd2}), 
\begin{align}
&(y+2)(d_{t,k} +c(y+2,t,k)) +3-g_{t,k}-g(y+2,t,k) +d(y+2,t,k) =\notag \\
& (y+2)(d_{t,k}+a(y+2,t,k,y)_\tau)+b(y+2,t,k,y)_\tau +3- \tau \notag
\end{align}
Hence
\begin{equation}\label{eqo+n2.0}
(y+2)(c(y+2,t,k)-a(y+2,t,k,y)_\tau )= -d(y+2,t,k) +b(y+2,t,k,y)_\tau +1\end{equation}
From (\ref{eqo+n2.0}) we get $a(y+2,t,k,y)_\tau \ge c(y+2,t,k)-1$. Since $c_{y+2,t,k} \ge g(y+2,t,k) +3 = \tau - g_{t,k}+4$, we get $a(y+2,t,k,y)_\tau \ge \tau-g_{t,k}+3$.

Now assume $x\ge y+4$. By Lemma \ref{+n2} we have $a(x,t,k,y)\ge a(y+2,t,k,y)$.
\end{proof}

 By Lemma \ref{o+n2.0} there is a non-special curve of degree $a(x,t,k,y)$ and genus $g-g_{t,k}$. We need this observation in the next statement.

 \quad {\bf{Assertion}} $N(x,t,k,y)$, $x\ge y$, $x\equiv y\pmod{2}$:  Set $e=1$ if $0\le b(x,t,k,y) \le a(x+2,t,k,y)-a(x,t,k,y) -1$ and $e=2$ if $b(x,t,k,y) \ge a(x+2,t,k,y)-a(x,t,k,y)$. There
 is a $6$-tuple
  $(X,Q,D_1,D_2,S_1,S_2)$ such that
 \begin{itemize}
 \item[(a)] $Q$ is a smooth quadric surface, $X = C_{t,k}\sqcup Y$, $Y$ is a smooth non-special curve of degree $a(x,t,k,y)$ and genus $g-g_{t,k}$
 and $Q$ intersects transversally
 $X$, with no line of $Q$ containing $\ge 2$ points of $X\cap Q$;
 \item[(b)] $D_1,D_2$ are different elements of $|\Oo _Q(1,0)|$, each of them containing one point of $Y\cap Q$, $S_i\subset D_i\setminus D_i\cap Y$, $1\le i \le 2$,
 and $\sharp (S_1)+\sharp (S_2) =b(x,t,k,y)$; $\pi _2(S_2) \subseteq \pi _2(S_1)$ and $\pi _2(S_e) \subset \pi _2(Y\cap (Q\setminus (D_1\cup D_2)))$;
$S_2=\emptyset$ if $e=1$, $\sharp (S_2) = b(x+2,t,k,y) -a(x+2,t,k,y)+a(x,t,k,y) +2$ if $e=2$;
 \item[({c})] $h^i(\Ii _{X\cup S_1\cup S_2}(x)) =0$, $i=0,1$.
 \end{itemize}

\begin{lemma}\label{+g1}
If $N(x,t,k,y)$ is true, then $N(x+2,t,k,y)$ is true.
\end{lemma}

\begin{proof}
We outline the modifications of the proof of Lemma \ref{+a2} needed to get Lemma \ref{+g1}. Let $e\in \{1,2\}$ (resp. $f\in \{1,2\}$) be the integer arising in $N(x,t,k,y)$ (resp. $N(x+2,t,k,y)$).  Take $(X,Q,D_1,D_2,S_1,S_2)$ satisfying $N(x,t,k,y)$. Set $w:= a(x+2,t,k,y)-a(x,t,k,y)$. 

\quad (a) Assume $e=2$. Set $z:= b(x+2,t,k,y)+2-w$. Since $b(x+2,t,k,y) \le x+1$, Lemma \ref{+n2} gives $z \le w-2$. Let $L_i\in |\Oo _Q(0,1)|$, $1\le i \le w-2$, be the lines such that 
$S_1 =D_1\cap (L_1\cup \cdots \cup L_{w-2})$ and $S_2 = D_2\cap (L_1\cup \cdots \cup L_z)$. Set $J:= D_1\cup D_2\cup (\bigcup _{i=1}^{w-2} L_i)$ and $\chi := \cup _{o\in S_1\cup S_2} \chi (o)$. Condition (b) gives $\sharp (L_i\cap Y) =1$ for all $i$. Condition (a) gives $C_{t,k}\cap J=\emptyset$. Hence $W:= X\cup J\cup \chi$ is a smoothable curve of degree $a(x+2,t,k,y)$
with $h^1(\Oo _W)=g$.

\quad (b) Assume $e=1$, i.e. assume $d(x+2,t,k,y)\le w-1$. Let $L_i\in |\Oo _Q(0,1)|$, $1\le i \le b(x,t,k,y)$, be the lines such that $S_1 =D_1\cap (L_1\cup \cdots \cup L_{b(x,t,k,y)})$. Take general lines $L_j\in |\Oo _Q(0,1)|$, $b(x,t,k,y)<j\le w-1$.
Set $J:= D_1\cup (\bigcup _{i=1}^{w-1} L_i)$ and $\chi := \cup _{o\in S_1} \chi (o)$. Condition (a) gives $C_{t,k}\cap J=\emptyset$. Hence $W:= X\cup J\cup \chi$ is a smoothable curve of degree $a(x+2,t,k,y)$
with $h^1(\Oo _W)=g$.\end{proof}

 \begin{lemma}
 $N(y+2,t,k,y)$ is true.
 \end{lemma}
 
 \begin{proof}
Use the proof of Lemma \ref{+a2} and Lemma \ref{+g1} starting with $(X,Q,D_1,D_2,S_1,S_2)$ satisfying $M(y,t,k)$ and quoting Lemma \ref{+n3} instead of Lemma \ref{+n2}.\end{proof}

 
\section{Proving Conjecture \ref{main}}\label{Sm}
 
In order to prove Theorem \ref{bingo} and Corollary \ref{cor}, first of all we notice that from the previous section we could deduce with a small effort the following two facts, but that
(as explained at the end of the introduction) they would not prove Theorem \ref{bingo} and Corollary \ref{cor}.  
 
For each integer $d$ such that $g-3 \le d \le d(m,g)_{\mathrm{max}}$ there exists a smooth and connected 
curve $X_1\subset \PP^3$ such that $\deg (X_1) =d$, $g(X)=g$, $h^1(\Oo _{X_1}(m-2)) =0$, $h^1(\Ii _{X_1}(m))=0$ and $h^1(N_{X_1}(-1))=0$. 

For each integer $d\ge d(m,g)_{\mathrm{min}}$ there exists a smooth and connected curve $X_2\subset \PP^3$ such that
$\deg (X_2) =d$, $g(X)=g$, $h^1(\Oo _{X_2}(m-2)) =0$,  $h^0(\Ii _{X_2}(m-1))=0$ and $h^1(N_{X_2}(-1))=0$.
 
Now fix an integer $d$ such that $d(m,g)_{\mathrm{min}} \le d \le d(m,g)_{\mathrm{max}}$. To prove Theorem \ref{bingo} for the pair $(d,g)$ it is sufficient to prove that we may find $X_1,X_2$
as above and with the additional condition that $X_1$ and $X_2$ are in the same irreducible component, $\Gamma$, of $\mathrm{Hilb}(\PP^3)$. If we prove this statement, then
by the semicontinuity theorem for cohomology (\cite[III.8.8]{hart}) we get $h^1(\Ii _X(m)) =0$ and $h^0(\Ii _X(m-1))=0$, hence we would conclude the proof for the pair $(d,g)$.
To get $X_1$ and $X_2$ in the same  irreducible component of $\mathrm{Hilb}(\PP^3)$ we need to rewrite the proofs of the previous section with a few improvements. But first we need to
distinguish between the case in which $d$ is very near to $d(m,g)_{\mathrm{min}}$ and the case in which $d$ is very near to $d(m,g)_{\mathrm{max}}$.
In the first case (say $d(m,g)_{\mathrm{min}} \le d\le d'$) we will modify the proof of the existence of $X_2$ with $h^0(\Ii _{X_2}(m-1))=0$ to get (for the same curve $X_2$) also $h^1(\Ii _{X_2}(m)) =0$.
If $d$ is very near to $d(m,g)_{\mathrm{max}}$ (say $d'' \le d \le d(m,g)_{\mathrm{max}}$) we will modify the proof of the existence of the curve $X_1$ to get a curve $X_1$ with $h^1(\Ii _{X_1}(m))=0$ and $h^0(\Ii _{X_1}(m-1))=0$. We use that $N(x,t,k,y)$ are true for $x=m-5,m-4,m-3,m-2$ (Lemma \ref{n9}).
 
Set $\varepsilon :=0$ if $m$ is odd and $\varepsilon :=1$ if $m$ is even.

\subsubsection{Near $d(m,g)_{\mathrm{min}}$}
 
 In this range the most difficult part is the proof of the existence of $X_2$. It is the construction of $X_2$ which says in which $W(t',k',d',b')$ we will try to find $X_1$.
 Recall that to get a curve $X_2$ with $h^0(\Ii _{X_2}(m-1)) =0$ we started with a curve $C_{t,t-\varepsilon}$ with $h^i(\Ii _{C_{t,t-\varepsilon}}(2t-1-\varepsilon)) =0$, where $t$ is the maximal integer $t>0$ such that such that $g_{t,t-\varepsilon} +g(2t+5-\varepsilon ,t,t-\varepsilon) \le g$. Set $k:=t-\varepsilon$. Recall that an element $W$
 of $U(t,k,a_d,b)$ has degree $d$ and $h^1(\Oo _W)=g$ if and only if  $b = g -g_{t,k}$ and $a_d = d -d_{t,k}$.
 The component $W(t',k',d',b')$ is the component $W(t,k,a_d,b)$, where
 $b = g -g_{t,k}$ and $a_d = d -d_{t,k}$. The curve $T$ satisfying $N(m-1,t,k,y)$ has $h^1(\Oo _T)=g$, $3$ connected components, $h^0(\Ii _T(m-1)) =b(m-1,t,k,y)$ and $h^1(\Ii _T(m-1)) =0$, hence $d> a(m-1,t,k,y)+d_{t,k}$. The minimum integer $d(m,g)_{\mathrm{min}}$ is $a(m-1,t,k,y)+d_{t,k}+1$, unless $b(m-1,t,k,y) \in \{m-2,m-1\}$ (in the latter case
 we have $d(m,g)_{\mathrm{min}} = a(m-1,t,k,y)+d_{t,k}+2$).

\quad (a)  We make the construction of Section \ref{horace} for the integer $m':= m-1 \equiv t+k-1 \pmod{2}$ and the integer $g$ (note that the numerology for $g$ in Theorem \ref{bingo} is such that
we may do the construction of Section \ref{horace} for $m':=m-1$ and the integer $g$). We get an integer $y\le m'-4 =m-5$ with $y\equiv t+k-1 \equiv 0 \pmod{2}$. Then for
all integers $x\ge y+2$ with $x\equiv y \pmod{2}$ we proved $N(x,t,k,y)$. Hence $N(m-5,t,k,y)$ and $N(m-3,t,k,y)$ are true (Lemma \ref{n9}). Since $d\ge d(m,g)_{\mathrm{min}}$,
we have $d > a(m-1,t,k,y)+d_{t,k}$, hence we want to add in a smooth quadric $Q$ a certain union of $d-a(m-3,t,k,y) -d_{t,k}$ lines. We write $C_t\cup C'_k$ for a general (but fixed in this construction) $C_{t,k}$, because we need to distinguish the two connected components of $C_{t,k}$, even
when $k=t$.

\quad (a1) Assume $d = d(m,g)_{\mathrm{min}}= a(m-1,t,k,y)+d_{t,k} +1$. Set $z:= d-a(m-3,t,k,y) -d_{t,k} = 1+ a(m-1,t,k,y) -a(m-3,t,k,y)$.  We need to modify
$N(m-3,t,k,y)$ in the following way.

\quad {\bf{Assertion}} $N'(m-3,t,k,y)$, $m-3\equiv y\pmod{2}$:  Set $e=1$ if $b(m-3,t,k,y) \le z-3$ and $e=2$ if $b(m-3,t,k,y) \ge z-2$. There
 is a $6$-tuple
  $(X,Q,D_1,D_2,S_1,S_2)$ such that
 \begin{itemize}
 \item[(a)] $Q$ is a smooth quadric surface, $X = C_t\sqcup C'_k\sqcup Y$, $Y$ is a smooth curve of degree $a(m-3,t,k,y)$ and genus $g-g_{t,k}$
 and $Q$ intersects transversally
 $X$, with no line of $Q$ containing $\ge 2$ points of $X\cap Q$;
 \item[(b)] $D_1,D_2$ are different elements of $|\Oo _Q(1,0)|$, $D_1\cap C_t\ne \emptyset$, $D_2\cap C_k\ne \emptyset$, $S_i\subset D_i\setminus D_i\cap (C_t\cup C'_k)$, $1\le i \le 2$,
 and $\sharp (S_1)+\sharp (S_2) =b(m-3,t,k,y)$; $\pi _2(S_2) \subseteq \pi _2(S_1)$, $\pi _2(S_e)\subset \pi _2(Y\cap (Q\setminus (D_1\cup D_2)))$;
$S_2=\emptyset$ if $e=1$, $\sharp (S_2) = b(m-3,t,k,y) -z +3$ if $e=2$;
 \item[({c})] $h^i(\Ii _{X\cup S_1\cup S_2}(m-3)) =0$, $i=0,1$.
 \end{itemize} 

 As in the proof of Lemma \ref{+a2} and Lemma \ref{+g1} we get $(X,Q,D_1,D_2,S_1,S_2)$, $X = C_t\sqcup C'_k\sqcup Y$ satisfying $N'(m-3,t,k,y)$; in the proof of Lemma \ref{+a2} we take $R_1$ containing a point of $C_t\cap Q$ instead of a point of $Y\cap Q$ and $R_2$ containing
a point of $C'_k\cap Q$ instead of a point of $Y\cap Q$.

\quad (a1.1) Assume $b(m-3,t,k,y)=0$. Take $D_0\in |\Oo _Q(1,0)|$ containing one point of $Y\cap Q$, $L_1\in |\Oo _Q(0,1)|$ containing a point of $C_t$, $L_2\in |\Oo _Q(0,1)|$
containing a point of $C'_k$ and general $L_i\in |\Oo _Q(0,1)|$, $3\le i\le z-1$. Set $J:= D_0\cup (\bigcup _{i=1}^{z-1} L_i)$. Since $X\cap (Q\setminus J)$ is a general subset
of $Q$ with cardinality $2d_{t,k} +2a(m-3,t,k,y)-3$,
we have $h^0(Q,\Ii _{Q\cap (X\cup J)}(m-1)) = h^0(Q,\Ii _{X\cap (Q\setminus J)}(m-2,m-z)) =0$ (use (\ref{eqd2}) for $x=m-3$, that $z= 1+a(m-1,t,k,y)-a(m-3,t,k,y)$ and that $b(m-1,t,k,y)\le m-2$). Since $\Res _Q(X\cup Y) = X$ and $h^0(\Ii _X(m-3)) =0$,
we have $h^0(\Ii _{X\cup J}(m-1)) =0$. The union $X\cup J$ is a nodal and connected smoothable curve of degree $d$ and arithmetic genus $g$ and $Y\cup J$ is a connected smoothable curve of degree $d-d_{t,k}$ and arithmetic genus $g-g_{t,k}-2\ge 26$. We may smooth $Y\cup J$ in a family of curves, all of them containing the two points $(C_t\cup C'_k)\cap J$. Call $E$ a general element of this smoothing. Since $\mathrm{Aut}(\PP^3)$ is $2$-transitive, we may see $E$ as a general non-special space curve of its degree and its genus $\ge 26$. By construction and Lemma \ref{aaa1} we have
$C_t\cup C'_k\cup E\in U(t,k,a_d,b)$ and $h^1(N_{C_t\cup C'_k\cup E}(-1)) =0$. By semicontinuity there is a smooth $X_2\in W(t,k,a_d,b)$ with 
$h^0(\Ii _{X_2}(m-1)) =0$ and $h^1(N_{X_2}(-1)) =0$.

\quad (a1.2) Assume $0 < b(m-3,t,k,y) \le z-3$. Hence $S_2=\emptyset$. We take $D_1$ and call $L_i\in |\Oo _Q(0,1)|$, $1\le i \le b(m-3,t,k,y)$, the elements of $|\Oo _Q(0,1)|$
such that $S_1 =D_1\cap (L_1\cup \cdots \cup L_{b(m-3,t,k,y)})$; note that each line $L_i$ contains a point of $Y\cap Q$. Take any $L_{b(m-3,t,k,y)+1}\in |\Oo _Q(0,1)|$ with $C'_k\cap L_{b(m-3,t,k,y)+1}\ne \emptyset$, any $L_{b(m-3,t,k,y)+2}\in |\Oo _Q(0,1)|$ with $Y\cap L_{b(m-3,t,k,y)+2} \ne \emptyset$, $L_{b(m-3,t,k,y)+2}\ne L_i$ for
$i\le b(m-3,t,k,y)$ and (if $b(m-3,t,k,y) <z-3$) take general $L_j\in |\Oo _Q(0,1)|$, $b(m-3,t,k,y)+3 \le j\le z-1$. Set $J:= D_1\cup (\bigcup _{i=1}^{z-1} L_i)$, $\chi := \cup _{o\in S_1} \chi (o)$
and $W:= X\cup J\cup \chi$. We have $\Res _Q(W) =X\cup S_1$ and thus $h^0(\Ii _{\Res _Q(W)}(m-3)) =0$. Since $W\cap Q$ is the union of $J$
and $2d_{t,k} +2a(m-3,t,k,y) -b(m-3,t,k,y)-3$ general points of $Q$ and $b(m-1,t,k,y) \le m-1$, (\ref{eqd2.1+}) gives $h^0(Q,\Ii _{W\cap Q}(m-1)) = h^0(Q,\Ii _{X\cap (Q\setminus J)}(m-2,m-z)) =0$. Thus
$h^0(\Ii _W(m-1)) =0$. We first deform $W$ to the union $F$ of $C_t\cup C'_k\cup D_1\cup Y\cup (\bigcup _{i=b(m-3,t,k,y)+1}^{z-1}L_i)$ and $b(m-3,t,k,y)$ disjoint lines $M_1,\dots ,M_{b(m-3,t,k,y)}$, each of them containing one point of $Y$. The union $F$
is a nodal and connected curve. Write $F =C_t\cup C'_k\cup G$. We have $\sharp (G\cap C_t) =\sharp (G\cap C'_k) =1$. Let $G'$ be a general smoothing of $G$ fixing
the $2$ points of $(C_t\cup C'_k)\cap G$.  $C_t\cup C'_k\cup G'\in U(t,k,a_d,b)$. By Lemma  \ref{aaa1} and
semicontinuity  there is a smooth $X_2\in W(t,k,a_d,b)$ with $h^0(\Ii _{X_2}(m-1)) =0$
and $h^1(N_{X_2}(-1)) =0$.

\quad (a1.3) Assume $b(m-3,t,k,y)\ge z-2$. Since $z= a(m-1,t,k,y)-a(m-3,t,k,y)+1$ and $b(m-3,t,k)) \le m-4$, Lemma \ref{+n2} gives $2(z-3) \ge b(m-3,t,k,y)$. Let $L_i\in |\Oo _Q(0,1)|$, $1\le i\le z-3$, be the lines such that $S_1=D_1\cap (\bigcup _{i=1}^{z-3} L_i)$ and $S_2:= D_2\cap (\bigcup _{i=1}^{w} L_i)$. Take $L_{z-2}\in |\Oo _Q(0,1)|$
containing one point of $Y\cap Q$ and different from the other lines $L_i$, $i\le z-3$. Set $J:= D_1\cup D_2\cup (\bigcup _{i=1}^{z-2} L_i)$, $\chi := \cup _{o\in S_1} \chi (o)$
and $W:= X\cup J\cup \chi$. We have $\Res _Q(W) =X\cup S_1\cup S_2$ and thus $h^0(\Ii _{\Res _Q(W)}(m-3)) =0$. Since $W\cap Q$ is the union of $J$
and $2d_{t,k} +2a(m,t,k,y) -w-3$ general points of $Q$ and $b(m-1,t,k,y) \le m-1$ (\ref{eqd2.1+}) gives $h^0(Q,\Ii _{W\cap Q}(m-1)) = h^0(Q,\Ii _{X\cap (Q\setminus J)}(m-2,m-z)) =0$. Thus
$h^0(\Ii _W(m-1)) =0$. We first deform $W$ to the union $F$ of $C_t\cup C'_k\cup D_1\cup D_2\cup Y\cup (\bigcup _{i=w+1}^{z-2}L_i)$ and $w$ disjoint lines $M_1,\dots ,M_{w}$, each of them containing one point of $Y$. The union $F$
is a nodal and connected curve. Write $F =C_t\cup C'_k\cup G$. We have $\sharp (G\cap C_t) =\sharp (G\cap C'_k) =1$. Let $G'$ be a general smoothing of $G$ fixing
the $2$ points of $(C_t\cup C'_k)\cap G$.  We have $C_t\cup C'_k\cup G'\in U(t,k,a_d,b)$. By Lemma  \ref{aaa1} and
semicontinuity  there is a smooth $X_2\in W(t,k,a_d,b)$ with $h^0(\Ii _{X_2}(m-1)) =0$
and $h^1(N_{X_2}(-1)) =0$.

\quad (a1.4) Assume  $d(m,g)_{\mathrm{min}} = a(m-1,t,k,y)+d_{t,k}+2$. We are in the set-up of step (a1.3) with the integer $z':= a(m-1,t,k,y)-a(m-3,t,k,y)+2$ instead of the integer $z:= a(m-1,t,k,y)-a(m-3,t,k,y)+1$.

\quad (a2) Assume $d> d(m,g)_{\mathrm{min}}$ and set $w:= d-d(m,g)_{\mathrm{min}}$. By step (a1)
there is a nodal curve
$E =C_t\cup C'_k\cup F\in U(t,k,a_d-w,b)$ with  $\sharp (C_t\cap F) =\sharp (C'_k\cap F)=1$, $C_t\cap D'_k=\emptyset$, $F$ 
and $h^0(\Ii _E(m-1))=0$. Take a general union $G$ of $F$ and $w$ lines, each of them meeting $F$ at exactly one point and quasi-transversally. By construction
$E':= C_t \cup C'_k \cup G$ is nodal and $C_t \cap G =C_t \cap F$, $C'_k \cap G = C'_k \cap F$.
Since $h^0(\Ii _E(m-1)) =0$ and $E'\supset E$, we have $h^0(\Ii _{E'}(m-1))=0$. We may smooth $G$ keeping fixed the points $C_t\cap F$ and $C'_k\cap F$,
because $\mathrm{Aut}(\PP^3)$ is $2$-transitive. Hence there is a non-special smooth curve $G''$ of  degree $d -d_{t,k}$ and genus $g -g_{t,k}$
with $C_t\cap G'' = C_t\cap F$, $C'_k\cap G'' = C'_k\cap F$ and which is a general member of a family with $F'$ as its special member and with $C_t\cup C'_k\cup G''$ nodal. By semicontinuity
we have $h^0(\Ii _{C_t\cup C'_k\cup G''}(m-1)) =0$. We have $C_t\cup C'_k\cup G''\in U(t,k,a_d,b)$.

\quad (b) Set $\alpha := t(t-2)$ if $k=t$ and $\alpha := t^2-3t+1$ if $k=t-1$. Fix a plane $H$, a smooth conic $D\subset H$ and general $C_{t,k}$. We have $D\cap C_{t,k} =\emptyset$ and $C_{t,k}\cap H$ is a general subset of $H$ with cardinality
$d_{t,k}$. Hence $h^0(H,\Ii _{H\cap (C_{t,k} \cup D)}(t+k)) = h^0(H,\Ii _{C_{t,k}\cap H}(t+k-1)) =\binom{t+k+1}{2} - d_{t,k} = \alpha$ and $h^1(H,\Ii _{H\cap (C_{t,k} \cup D)}(t+k))=0$.
Then we continue the construction from the critical value $t+k$ to the critical value $t+k+2$, then to the critical value $t+k+4$, and so on up to the critical value $m-2$; in each step, say to arrive at the  critical value $x$ from a curve $A'$ and a set $S'$ with $h^1(\Ii _{A'\cup S'}(x-2)) =0$ and $h^0(\Ii _{A'\cup S'}(x-2)) =\alpha$ and $0\le \sharp (S')\le x-3$
(and so $\sharp (S') =\binom{x+1}{3} -(x-2)\deg (A')-3+g -\alpha$ ; we have bijectivity inside $Q$ and get a curve $A''$ and a set $S''$ with $h^1(\Ii _{A''\cup S''}(x)) =0$ and $h^0(\Ii _{A''\cup S''}(x)) \le \alpha$. In the last step we also need to connect the connected components of the curve and get an element $B\in U(t,k,a',b)$ for some $a'$; we need to check that at each step the numerical conditions are satisfied.   Call $(X,Q,D_1,D_2,S_1,S_2)$ the curve we get for $\Oo _{\PP^3}(m-2)$ and either $e=1$ or $e=2$. Set $S:= S_1\cup S_2$
and $\alpha ':= \sharp (S)$. We have $0\le \alpha '\le m-3$. Since $S$ is a union of connected components of $X\cup S$, the restriction map $H^0(\Oo _{X\cup S}(m-2)) \to H^0(\Oo _X(m-2)) $ is surjective and its kernel
has dimension $\sharp (S) $. Since $h^1(\Ii _{X\cup S}(m-2)) =0$, we have $h^1(\Ii _X(m-2)) =0$ and $h^0(\Ii _X(m-2) =\alpha +\alpha ' \le \alpha +m-3$. We cover in this way the integers $d$ such that $\binom{m+3}{3} +g-1-dm \ge \alpha +m-3$. Hence we cover all $d$ such that $d(m,g)_{\mathrm{max}} -d \ge 1+ \lfloor \alpha /m\rfloor$. If $t\le m/4$ we have $\alpha /m \le m/4$.

\subsubsection{Near $d(m,g)_{\mathrm{max}}$}

In this range the most difficult part is the existence of $X_1$ with $h^1(\Ii _{X_1}(m)) =0$ and it is this part which dictates the component $W(t',k',a',b')$ in 
which we will find both $X_1$ and $X_2$. We stress that the integers $t, k$ introduced in this subsection are not the same as in the previous one and hence also $y$ may be different.

\quad (a) In this step we prove the existence of $X_1$. We start with the maximal integer $k$ such that $g_{k+1-\varepsilon,k} +g(2k+6-\varepsilon, k+1-\varepsilon,k)\le g$ and set $t:= k+1-\varepsilon$. We use $N(x,t,k,y)$.
In particular we have $N(m-4,t,k,y)$ and $N(m-2,t,k,y)$. Set $a_d:= d-d_{t,k}$ and $b:= g-g_{t,k}$.
In this step we prove the existence of $A\in U(t,k,a_d,b)$ with $h^1(\Ii _A(m)) =0$, hence by semicontinuity the existence of $X_1\in W(t,t-1,a_d,b)$ with $h^1(\Ii _{X_1}(m))=0$.
Set $z:= d-a(m-2,t,k,y)-d_{t,k}$. We write $C_t\cup C'_k$ for a general (but fixed in this construction) $C_{t,k}$, because we need to distinguish the two connected components, even
when $k=t$. Recall that we have (\ref{eqi1}).

\quad (a1) Assume $d=d(m,g)_{\mathrm{max}}$. Let $T$ be any curve satisfying $N(m,t,k,y)$. We have $\deg (T) =d_{t,k} +a(m,t,k,y)$, $h^1(\Oo _T)=g$, $h^1(\Oo _T(m))=0$, $T$ has $3$ connected components,
$h^1(\Ii _T(m)) =0$ and $h^0(\Ii _T(m)) =b(m,t,k,y)$. By (\ref{eqi1}) we have $d = a(m,t,k,y)+d_{t,k}$ if $b(m,t,k,y) \le m-3$ and $d=a(m,t,k,y)+d_{t,k}+1$ if $m-2\le b(m,t,k,y) \le m-1$.
Hence $a(m,t,k,y) -a(m-2,t,k) \le z\le a(m,t,k,y) -a(m-2,t,k,y)+1$.  Call $\eta$ the difference between the right hand side and the left hand side of (\ref{eqi1}).

 \quad {\bf{Assertion}} $N''(m-2,t,k,y)$, $m\equiv y\pmod{2}$:  Set $e=1$ if $b(m-2,t,k,y) \le z-3$ and $e=2$ if $b(x,t,k,y) \ge z-2$. There
 is a $6$-tuple
  $(X,Q,D_1,D_2,S_1,S_2)$ such that
 \begin{itemize}
 \item[(a)] $Q$ is a smooth quadric surface, $X = C_t\sqcup C'_k\sqcup Y$, $Y$ is a smooth curve of degree $a(m-2,t,k,y)$ and genus $g-g_{t,k}$
 and $Q$ intersects transversally
 $X$, with no line of $Q$ containing $\ge 2$ points of $X\cap Q$;
 \item[(b)] $D_1,D_2$ are different elements of $|\Oo _Q(1,0)|$, $D_1\cap C_t\ne \emptyset$, $D_2\cap C'_k\ne \emptyset$, $S_i\subset D_i\setminus D_i\cap (C_t\cup C'_k)$, $1\le i \le 2$,
 and $\sharp (S_1)+\sharp (S_2) =b(x,t,k,y)$; $\pi _2(S_2) \subseteq \pi _2(S_1)$ and $\pi _2(S_e)\subset \pi _2(Y\cap (Q\setminus (D_1\cup D_2)))$;
$S_2=\emptyset$ if $e=1$, $\sharp (S_2) = b(m-2,t,k,y) -z +2$ if $e=2$;
 \item[({c})] $h^i(\Ii _{X\cup S_1\cup S_2}(x)) =0$, $i=0,1$.
 \end{itemize}

 As in the proof of Lemma \ref{+a2} and Lemma \ref{+g1} we get $(X,Q,D_1,D_2,S_1,S_2)$, $X = C_t\sqcup C'_k\sqcup Y$ satisfying $N''(m-2,t,k,y)$; in the proof of Lemma \ref{+a2} we take $R_1$ containing a point of $C_t\cap Q$ instead of a point of $Y\cap Q$ and $R_2$ containing
a point of $C'_k\cap Q$ instead of a point of $Y\cap Q$.

 \quad (a1.1) Assume $b(m-2,t,k,y)=0$. Take $z-1$ distinct lines $L_i\in |\Oo _Q(0,1)|$, $1\le i\le z-1$, such that $L_i\cap C_t=\emptyset$ for all $i$, $L_i\cap C'_k \ne \emptyset$
 if and only if $i=1$ and $L_i\cap Y\ne \emptyset$ if and only if $i=2$. Set $J:= D_1\cup (\bigcup _{i=1}^{z-1} L_i)$. Since $X\cap (Q\setminus J)$ is a general subset
of $Q$ with cardinality $2d_{t,k} +2a(m-3,t,k,y)-3$,
we have $h^1(Q,\Ii _{Q\cap (X\cup J)}(m)) = h^1(Q,\Ii _{X\cap (Q\setminus J)}(m-1,m+1-z)) =0$ (use the generality
of $X\cap (Q\setminus J)$ and the difference between (\ref{eqi1}) and the case $x:= m-2$ of (\ref{eqd2}), which gives an upper bound for
$\sharp (X\cap (Q\setminus J))$; we get an equality if and only if $\eta =0$, i.e. $b(m,t,k,y) =m-2$ and $d= a(m,t,k,y)+d_{t,k}+1$). Since $\Res _Q(X\cup J) = X$ and $h^1(\Ii _X(m-2)) =0$,
we have $h^1(\Ii _{X\cup J}(m)) =0$. The union $X\cup J$ is a nodal and connected smoothable curve of degree $d$ and arithmetic genus $g$ and $Y\cup J$ is a smooth and connected
curve of degree $d-d_{t,k}$ and arithmetic genus $g-g_{t,k}-2\ge 26$. We may smooth $Y\cup J$ in a family of curves, all of them containing the two points $(C_t\cup C'_k)\cap J$. Call $E$ a general element of this smoothing. Since $\mathrm{Aut}(\PP^3)$ is $2$-transitive, we may see $E$ as a general non-special space curve of its degree and its genus $\ge 26$. By construction and Lemma \ref{aaa1} we have
$C_t\cup C'_k\cup E\in U(t,k,a_d,b)$ and $h^1(N_{C_t\cup C'_k\cup E}(-1)) =0$. By semicontinuity there is a smooth $X_1\in W(t,k,a_d,b)$ with $h^1(\Ii _{X_1}(m)) =0$
and $h^1(N_{X_1}(-1)) =0$.

\quad (a1.2) Assume $0 < b(m-2,t,k,y) \le z-3$.  Hence $S_2=\emptyset$. We take $D_1$ and call $L_i\in |\Oo _Q(0,1)|$, $1\le i \le b(m-2,t,k,y)$, the elements of $|\Oo _Q(0,1)|$
such that $S_1 =D_1\cap (L_1\cup \cdots \cup L_{b(m-2,t,k,y)})$; note that each line $L_i$ contains a point of $Y\cap Q$. Take any $L_{b(m-2,t,k,y)+1}\in |\Oo _Q(0,1)|$ with $C'_k\cap L_{b(m-2,t,k,y)+1}\ne \emptyset$, any $L_{b(m-2,t,k,y)+2}\in |\Oo _Q(0,1)|$ with $Y\cap L_{b(m-2,t,k,y)+2} \ne \emptyset$, $L_{b(m-2,t,k,y)+2}\ne L_i$ for
$i\le b(m-2,t,k,y)$ and (if $b(m-2,t,k,y) <z-3$) take general $L_j\in |\Oo _Q(0,1)|$, $b(m-2,t,k,y)+3 \le j\le z-1$. Set $J:= D_1\cup (\bigcup _{i=1}^{z-1} L_i)$, $\chi := \cup _{o\in S_1} \chi (o)$
and $W:= X\cup J\cup \chi$. We have $\Res _Q(W) =X\cup S_1$ and thus $h^1(\Ii _{\Res _Q(W)}(m-2)) =0$. Since $\eta \ge 0$, (\ref{eqi1}) and the case $x=m-2$
of (\ref{eqd2.1+}) give $2d_{t,k} +2a(m,t,k,y) -b(m-2,t,k,y)-3 =m(m+3-z) -\eta \le h^0(Q,\Oo _Q(m-2,m+2-z))$. Since $W\cap Q$ is the union of $J$
and $2d_{t,k} +2a(m,t,k,y) -b(m-2,t,k,y)-3$ general points of $Q$, we have $h^1(Q,\Ii _{W\cap Q}(m)) = h^1(Q,\Ii _{X\cap (Q\setminus J)}(m-1,m+1-z)) =0$. Thus
$h^1(\Ii _W(m)) =0$. We first deform $W$ to the union $F$ of $C_t\cup C'_k\cup D_1\cup Y\cup (\bigcup _{i=b(m-3,t,k,y)+1}^{z-1}L_i)$ and $b(m-3,t,k,y)$ disjoint lines $M_1,\dots ,M_{b(m-3,t,k,y)}$, each of them containing one point of $Y$. The union $F$
is a nodal and connected curve. Write $F =C_t\cup C'_k\cup G$. We have $\sharp (G\cap C_t) =\sharp (G\cap C'_k) =1$. Let $G'$ be a general smoothing of $G$ fixing
the $2$ points of $(C_t\cup C'_k)\cap G$.  $C_t\cup C'_k\cup G'\in U(t,k,a_d,b)$. By Lemma  \ref{aaa1} and
semicontinuity  there is a smooth $X_2\in W(t,k,a_d,b)$ with $h^1(\Ii _{X_2}(m)) =0$
and $h^1(N_{X_2}(-1)) =0$.

\quad (a1.3) Assume $b(m-2,t,k,y)  \ge z-2$. Since $z\ge a(m,t,k,y)-a(m-2,t,k)$ and $b(m-2,t,k,y) \le m-3$, the case $x=m-2$ of Lemma \ref{+n2} gives $2(z-3)\ge b(m-2,t,k,y)$.
Set $w:= b(m-2,t,k)-z+3$.
Let $L_i\in |\Oo _Q(0,1)|$, $1\le i \le z-3$, be the line such that $S_1=D_1(\bigcup _{i=1}^{z-3} L_i)$ and $S_2:= D_2\cap (\bigcup _{i=1}^{w} L_i)$. Let $L_{z-2}\in |\Oo _Q(0,1)|$
be a line with $L_{z-2}\ne L_i$ for any $i\ne z-2$ and $L_{z-2}\cap Y \ne \emptyset$. Note that $L_j\cap Y\ne \emptyset$ if and only if either $j\le w$ or $j=z-2$.
Set $J:= D_1\cup D_2\cup (\bigcup _{i=1}^{z-2} L_i)$, $\chi := \cup _{o\in S_1\cup S_2} \chi (o)$
and $W:= X\cup J\cup \chi$ and continue as in the last step. 

\quad (a2) Assume $d<d(m,g)_{\mathrm{max}}$ We have $\eta \ge m(d(m,g)_{\mathrm{max}}-d)\ge m$ and in particular $\eta \ge m\ge  b(m-2,t,k,y)+2$. To prove the existence of $X_1$ in this component we only need that $z\ge 3$, i.e. that
$d\ge a_{m-2,t,k,y}+d_{t,k}+3$, which is true because $1+(m-1)d -g \ge \binom{m+2}{3}$ and $(m-1)(a(m-2,t,k,y)+d_{t,k}) +3-g = \binom{m+1}{2} -a(m-2,t,k) -d_{t,k} +b(m-2,t,k,y)
\ge 3m$. Take $(X,Q,D_1,D_2,S_1,S_2)$ satisfying $N(m-2,t,k,y)$ with $X = C_t\sqcup C'_k\sqcup Y$ and throw away $D_1$, $D_2$, $S_1$ and $S_2$. Fix $D\in |\Oo _Q(1,0)|$ containing one point of $Y\cap Q$ and $z-1$ distinct lines $L_i\in |\Oo _Q(0,1)|$ with $L_i\cap Y=\emptyset$ for all $i$,
$L_i\cap C_t\ne \emptyset$ if and only if $i=1$ and $L_i\cap C'_k\ne \emptyset$ if and only if $i=2$. Set $J:= D\cup (\bigcup _{i=1}^{z-1} L_i)$ and $W:= X\cup J$.
As in the previous steps it is sufficient to prove that $h^1(\Ii _W(m)) =0$. We have $\Res _Q(W) =X$ and thus $h^1(\Ii _{\Res _Q(W)}(m-2))=0$. Hence it is sufficient to prove
that $h^1(Q,\Ii _{W\cap Q}(m)) =0$. We have $h^1(Q,\Ii _{Q\cap W}(m)) = h^1(Q,\Ii _{X\cap (Q\setminus J)}(m-1,m+1-z))$. Since $X\cap Q$ is general in $Q$,
it is sufficient to prove that $\sharp (X\cap (Q\setminus J)) \le m(m+2-z)$. We have  $\sharp (X\cap (Q\setminus J)) =2d_{t,k} +2a(m-2,t,k,y)-3$.
By the definition of $\eta$ and (\ref{eqd2}) for $x=m-2$ we have $2d_{t,k} +2a(m-2,t,k,y)-3 =m(m+2-z) +b(m-2,t,k,y)+2 -\eta \le m(m+2-z)$.

\quad (b) In this part we get the existence of $A\in U(t,k,a_d,b)$ with $h^0(\Ii _A(m-1)) =0$, $\deg (A)=d$ and $p_a(A)=g$, hence by semicontinuity 
the existence of $X_2\in W(t,k,a_d,b)$ with $h^0(\Ii _{X_2}(m-1)) =0$. We have $h^i(\Ii _{C_{t,k}}(t+k-1)) =0$, $i=0,1$ and $m-1\equiv t+k \pmod{2}$. Fix
a plane $H$. Let $c$ be the maximal integer such that $\binom{t+k+2-c}{2} \le d_{t,k}$. Let $E\subset H$ be a general linear projection of a general smooth and rational degree $c$ curve $E'\subset \PP^3$. The curve $E$ is nodal and it has $(c-1)(c-2)/2$ singular points. 
Set $\chi := \cup _{p\in \mathrm{Sing}(E)} \chi ({p})$. The union $E\cup \chi$ is the flat limit
of a family of degree $c$ smooth rational curves in $\PP^3$ (\cite[Fig. 11 at p. 260]{hart}. Hence to prove that a general union of some $C_{t,k}$ and a smooth rational curve of degree $c$ is contained in no surface of degree $t+k$ it is sufficient to prove that $h^0(\Ii _{C_{t,k}\cup E\cup \chi}(t+k) =0$ for a general $C_{t,k}$. Thus it is sufficient to prove
that  $h^0(\Ii _{C_{t,k}\cup E}(t+k)) =0$ for a general $C_{t,k}$. For a general $C_{t,k}$ we have $C_{t,k}\cap E=\emptyset$ and $C_{t,k}\cap H$ is a general subset of $H$
with cardinality $d_{t,k}$. By definition $c$ is the minimal positive integer such that $h^0(H,\mathcal {I}_{C_{t,k}\cap H}(t+k-c)) =0$. 
Set $\beta = h^0(\Oo _{C_{t,k}\cup E\cup \chi}(t+k)) -\binom{t+k+3}{3}$. Since $\binom{t+k+2-c}{2} -\binom{t+k-1}{2} = t+k+1-c$,
we have $\beta \le (c-1)(c-2)/2 +t+k+1-c$. Then we continue from the critical value $t+k$ to the critical value $t+k+2$ and so on.

At the end we obtain some $B\in U(t,k,a_d,b)$ with $h^0(\Ii _B(m-1)) =0$ if $1+d(m-1)-g \ge \binom{m+2}{3}+\beta$. In particular it is sufficient to assume $d\ge d(m,g)_{\mathrm{min}} +\lceil \beta/(m-1)\rceil$. We have $c \sim \sqrt{2}t$, because $\deg (C_{t,k}) \sim t^2$ and $\binom{t+k+2}{2} \sim 2t^2$. Hence $\beta \sim (c-1)(c-2)/2 \sim t^2$.
Since $t\le m/4$, it is sufficient to have roughly $d\ge d(m,g)_{\mathrm{min}} +m/4$. 

\begin{lemma}\label{n9}
Fix $t$ and $k\in \{t-1,t\}$ such that $y\equiv t+k-1 \pmod{2}$
and let $g_{t,k} +g(t+k+5,t,k) \le g \le -1 + g_{t+1,k+1} + g(t+k+7,t+1,k+1)$.
Then we have $y \le \sqrt{20} t-1$. In particular, if $t \ge \lfloor m/\sqrt{20}\rfloor -5$ then $y \le m-6$. 
\end{lemma}

\begin{proof}
We have $g_{t+1,k+1} - g_{t,k} = 2t^2-2$ if $k=t$ and $g_{t+1,k+1} - g_{t,k} = 2t^2-2t-1$ if $k=t-1$. For all integer $x \ge t+k+1$ such that
$x\equiv t+k+1 \pmod{2}$ we have $c(x,t,k) -c(x-2,t,k)) \ge (x+2)/2$ (Lemma \ref{aa+1}). Remark \ref{aa+0} gives $c(t+k+1,t,k) =k+3$.
By the definition of $y$, we have $y \ge k+t+5$ and $g \ge g_{t,k} + g(y,t,k) = g_{t,k} + c(y,t,k)-3(y-t-k-1)/2-3 
\ge g_{t,k} -3(y-t-k-1)/2+k + \sum _{i=1}^{(y-t-k-1)/2}(c(t+k+1+2i,t,k)-c(t+k+1+2i-2,t,k)) \ge  g_{t,k} -3(y-t-k-1)/2 + k + (t+k+y+7)(y-t-k-1)/8$. 
On the other hand, we have $g \le -1 + g_{t+1,k+1} + g(t+k+7,t+1,k+1) \le -1 + g_{t+1,k+1} + 3(t+k+7)$. 
Hence we get $(t+k+y+7)(y-t-k-1)/8 \le g_{t+1,k+1} - g_{t,k} + 3(y-t-k-1)/2 -k -1 + 3(t+k+7)$ and in particular
$(y+1)^2 \le 20 t^2$.
\end{proof}

\vspace{0.4cm}
\begin{proof}[Proof of Theorem \ref{bingo}:] We fix the integer $g$ and we perform the above construction in both the odd and the even case, by taking either 
$k=t$ or $k=t-1$. We have $h^1(\Oo(C_{t,k}(t-1)=0$, hence we get $h^1(\Oo(C_X(t-1)=0$ by a repeated application of Mayer-Vietoris and semicontinuity. 
For every $t \ge 27$ such that $g \ge g_{t+3,k+3} \ge g_{t,k} +g(t+k+5,t,k)$ 
we get an integer $y\equiv t+k-1$ such that the statement of Theorem~\ref{bingo} holds for every $m \ge y+6$ with $m\equiv y \pmod{2}$. By Lemma \ref{n9}, the condition $m \ge y+6$ is satisfied for every $t \ge \lfloor m/\sqrt{20}\rfloor -5$, hence we obtain our statement 
for every $g$  with $2g_{30} = 17052 \le g \le \varphi (m)$.\end{proof}

\vspace{0.4cm}
\begin{proof}[Proof of Corollary \ref{cor}:] Let $m$ be the minimal non-negative integer such that 
\begin{equation}\label{eqtt1}
md+1-g \le \binom{m+3}{3}
\end{equation}
The minimality of $m$ gives
\begin{equation}\label{eqtt2}
(m-1)d +1-g > \binom{m+2}{3},
\end{equation}
in particular $d \ge \frac{(m+2)(m+1)m}{6(m-1)} \ge \frac{m^2}{6}$.
From (\ref{eqtt1}) and (\ref{eqtt2}) we get $d  \le \binom{m+2}{2}$. Since $g\le Kd^{3/2}- 6 \varepsilon d $, we have 
\begin{eqnarray*}
g &\le& \frac{2}{3} \left(\frac{1}{10} \right)^{3/2}  \binom{m+2}{2}^{3/2} - 6 \varepsilon d \\
&\le& \frac{2}{3} \left(\frac{1}{10} \right)^{3/2} \left(\frac{(m+2)^2}{2}\right)^{3/2}- 6 \varepsilon d \\
&\le& \frac{2}{3} \left(\frac{1}{20} \right)^{3/2} (m+2)^3 - \varepsilon m^2 \le \varphi (m) 
\end{eqnarray*}
(notice that the coefficients of $m^3$ are controlled by our choice of $K$ and the coefficients of $m^2$ are controlled by our choice of $\varepsilon$). 
Since $g \le \varphi (m)$, Theorem \ref{bingo} covers all degrees $d_0$ in the interval $d(m,g)_{\mathrm{min}}\le  d_0 \le d(m,g)_{\mathrm{max}}$. 
In order to check that $d$ is in this interval, just notice that $d \ge d(m,g)_{\mathrm{min}}$ by (\ref{eqtt2}) and $d\le d(m,g)_{\mathrm{max}}$ by (\ref{eqtt1}).\end{proof}

\end{document}